\documentclass[10pt,twoside]{amsart}

\usepackage{amssymb}
\usepackage{amscd}
\usepackage{amsmath}
\usepackage{amsrefs}
\usepackage{xypic}
\usepackage[OT2,T1]{fontenc}

\DeclareSymbolFont{cyrletters}{OT2}{wncyr}{m}{n}
\DeclareMathSymbol{\Sha}{\mathalpha}{cyrletters}{"58}

\newtheorem{prop}{Proposition}[section]
\newtheorem{lemma}{Lemma}
\newtheorem{corollary}{Corollary}
\newtheorem{conjecture}{Conjecture}
\newtheorem{definition}{Definition}
\theoremstyle{remark}
\newtheorem{remark}{Remark}

\newcommand{\ba}{\mathbb A}
\newcommand{\bq}{\mathbb Q}

\newcommand{\bz}{\mathbb Z}
\newcommand{\br}{\mathbb R}
\newcommand{\bc}{\mathbb C}
\newcommand{\bg}{\mathbb G}

\newcommand{\bF}{\mathbb F}

\newcommand{\co}{\mathcal O}
\newcommand{\cm}{\mathcal M}

\newcommand{\fp}{\mathfrak p}
\newcommand{\fq}{\mathfrak q}

\newcommand{\F}{\mathcal F}
\newcommand{\X}{\mathcal X}
\newcommand{\Y}{\mathcal Y}

\newcommand{\mydet}{\mathrm{det}}
\newcommand{\tr}{\tilde{\mathbb R}}
\newcommand{\Ar}{{\mathrm {ar}}}

\DeclareMathOperator{\Frob}{Fr}
\DeclareMathOperator{\ord}{ord}
\DeclareMathOperator{\Spec}{Spec}
\DeclareMathOperator{\rank}{rank}

\DeclareMathOperator{\trace}{Tr}
\DeclareMathOperator{\Pic}{Pic}

\DeclareMathOperator{\Ind}{Ind}

\DeclareMathOperator{\Br}{Br}
\DeclareMathOperator{\im}{im}
\DeclareMathOperator{\Gal}{Gal}

\DeclareMathOperator{\id}{id}
\DeclareMathOperator{\Hom}{Hom}

\DeclareMathOperator{\cok}{coker}
\DeclareMathOperator{\Ab}{Ab}
\DeclareMathOperator{\Set}{Set}
\DeclareMathOperator{\Sm}{Sm}
\DeclareMathOperator{\Schemes}{Sch}

\begin{document}

\title{Special values of the Zeta function of an arithmetic surface}

\author{Matthias Flach}
\author{Daniel Siebel}
\address{Dept. of Mathematics 253-37\\California Institute of Technology\\Pasadena CA 91125\\USA}
\begin{abstract}{We study the special value conjecture for the Zeta function of a proper regular arithmetic scheme $\X$ introduced in \cite{Flach-Morin-16} in the case $n=1$. We compute the correction factor $C(\X,1)$ left unspecified in \cite{Flach-Morin-16}, thereby developing some results on the eh-topology introduced by Geisser \cite{Geisser06}. We then specialize further to the case where $\X$ is an arithmetic surface and show that the conjecture of \cite{Flach-Morin-16} is equivalent to the Birch and Swinnerton-Dyer Conjecture.}\end{abstract}

\maketitle
\section{Introduction}

Let $\X$ be a regular scheme of dimension $d$, proper over $\Spec(\bz)$. In \cite{Flach-Morin-16}[Conj. 5.11, 5.12] the first author and Morin have formulated a conjecture on the vanishing order and the leading Taylor coefficient of the Zeta-function $\zeta(\X,s)$ of $\X$ at integer arguments $s=n\in\bz$ in terms of what we call Weil-Arakelov cohomology complexes. More specifically, our conjecture involves a certain invertible $\bz$-module
\[ \Delta(\X/\bz,n):={\det}_\bz R\Gamma_{W,c}(\mathcal{X},\mathbb{Z}(n))\otimes_\bz {\det}_\bz R\Gamma(\X_{Zar},L\Omega_{\X/\bz}^{<n})\]
which can be attached to the arithmetic scheme $\X$ under various standard assumption (finite generation of \'etale motivic cohomology in a certain range being the most important). Under these assumptions one has a natural trivialization
\begin{equation}\lambda_\infty:\br\xrightarrow{\sim}\Delta(\X/\bz,n)\otimes_\bz\br\label{ourtheta}\end{equation}
and we conjecture
$$\lambda_{\infty}(\zeta^*(\mathcal{X},n)^{-1}\cdot C(\mathcal{X},n)\cdot\mathbb{Z})= \Delta(\mathcal{X}/\mathbb{Z},n)
$$
which determines the leading coefficient $\zeta^*(\mathcal{X},n)\in\br$ up to sign (all identities in this paper involving leading coefficients should be understood up to sign). Here $C(\X,n)\in\bq^\times$ is a certain correction factor, defined as a product over its $p$-primary parts and where the definition of each $p$-primary part involves $p$-adic Hodge theory. It is easy to see that $C(\X,n)=1$ for $n\leq 0$ and one also has $C(\X,n)=1$ for $\X$ of characteristic $p$. In \cite{Flach-Morin-18}[Rem. 5.2] it was then suggested that in fact $C(\X,n)$ has the simple form
\begin{equation}\label{explicitcorrectingfactor}
C(\X,n)^{-1}=\prod_{ i\leq n-1;\, j}(n-1-i)!^{(-1)^{i+j}\mathrm{dim}_{\bq}H^j(\X_{\bq},\Omega^i)}
\end{equation}
for any $n\in \bz$ and any $\X$. This formula is corroborated by the computation of $C(\Spec(\co_F),n)$ in \cite{Flach-Morin-16}[Prop. 5.34] for a number field $F$ all of whose completions $F_v$ are absolutely abelian. %In forthcoming work we shall show that with this formula our special value conjecture is compatible with the functional equation of $\zeta(\X,s)$. With formula (\ref{explicitcorrectingfactor}) we make a precise conjecture that determines $\zeta^*(\X,n)\in\br$ up to sign and is independent of $p$-adic Hodge theory.

In the present article we focus on the case $n=1$ and then specialize further to arithmetic surfaces. We give more evidence for formula (\ref{explicitcorrectingfactor}) by proving that it holds for $n=1$, i.e. that $C(\X,1)=1$, for arbitrary $\X$. If $\X$ is connected, flat over $\Spec(\bz)$ and of dimension $d=2$ we call $\X$ an arithmetic surface. Denoting by $F$ the algebraic closure of $\bq$ in the function field of $\X$, the structural morphism $\X\to\Spec(\bz)$ factors through a morphism
\begin{equation} f:\X\to\Spec(\co_F)=:S \label{fdef}\end{equation}
with $f_*\co_\X=\co_S$. We show that all the assumptions entering into our conjecture are satisfied for $n=1$ if $\X$ has finite Brauer group, or equivalently the Jacobian $J_F$ of $\X_F$ has finite Tate-Shafarevich group. To give an idea of what the conjecture says concretely, recall that the Zeta function of an arithmetic surface factors
\begin{equation}\zeta(\X,s)=\frac{\zeta(H^0,s)\zeta(H^2,s)}{\zeta(H^1,s)}=\frac{\zeta_F(s) \zeta_F(s-1)}{\zeta(H^1,s)}\label{factorization}\end{equation}
where $\zeta(H^i,s)$ should be viewed as the Zeta function of a relative $H^i$ of $f$ in the sense of a motivic (i.e. perverse) $t$-structure, and $\zeta(H^1,s)$ differs from the Hasse-Weil L-function of $J_F$ by finitely many Euler factors. Our conjecture is equivalent to the statements
\begin{equation} \ord_{s=1}\zeta(H^1,s)=\rank_\bz\Pic^0(\X)\label{vanishingatone}\end{equation}
and
\begin{equation} \zeta^*(H^1,1)=\frac{\#\Br(\overline{\X})\cdot\delta^2\cdot\Omega(\X)\cdot R(\X)}{(\#(\Pic^0(\X)_{tor}/\Pic(\co_F)))^2}\cdot\prod\limits_{\text{$v$ real}}\frac{\#\Phi_v}{\delta_v'\delta_v}\label{BSDforX}\end{equation}
where $\Pic^0(\X)$ is the kernel of the degree map on $\Pic(\X)$, $R(\X)$ is the regulator of a certain intersection pairing on $\Pic^0(\X)$ and $\Omega(\X)$ is the determinant of the period isomorphism between the finitely generated abelian groups $H^1(\X(\bc),2\pi i\cdot\bz)^{G_\br}$ and $H^1(\X,\co_\X)$.  The integer $\delta$ is the index of $\X_F$, i.e. the g.c.d. of the degrees of all closed points. Furthermore, $\Phi_v=J_F(F_v)/J_F(F_v)^0$ is the group of components of the group  of $F_v$-rational points of $J_F$ and $\delta_v'$, resp. $\delta_v$, are the period, resp. index of $\X_{F_v}$ over $F_v$. The group $\Br(\overline{\X})$ coincides with $\Br(\X)$ if $F$ has no real places and differs from $\Br(\X)$ by a $2$-torsion group in general. The group $\Br(\overline{\X})$ is naturally self-dual and the quantity $\#\Br(\overline{\X})\delta^2$ will also arise as the cardinality of a naturally self-dual group in our computation (one could call this group the $H^1$-part of the Brauer group, see Lemma \ref{h1br} below).

We then prove that our conjecture (\ref{vanishingatone}) and (\ref{BSDforX}) is equivalent to the Birch and Swinnerton-Dyer conjecture for the Jacobian of $\X_F$, provided that our intersection pairing agrees with the Arakelov intersection pairing. This result was shown if $f$ is smooth in \cite{Flach-Morin-16}[Thm. 5.27] but only rather indirectly via compatibility of both conjectures with the Tamagawa number conjecture. Here we give a direct proof without assumptions on $f$ (such as existence of a section which would simplify the proof considerably). If $\X$ is a smooth projective surface over a finite field, fibred over a curve $f:\X\to S$, the equivalence of the Birch and Swinnerton-Dyer formula for the Jacobian of the generic fibre with a special value conjectures for the Zeta function $\zeta(H^2_{abs},s)$ of the absolute $H^2$-motive of $\X$ has been much studied in the literature, going back to the article of Tate \cite{tate66}, see also \cite{gordon}, \cite{LLR}, \cite{LLRcor}. At least for suitable choices of $f$, the two functions $\zeta(H^1,s)$ and $\zeta(H^2_{abs},s)$ only differ by a very simple Euler factor (see Remark \ref{tate} below), and it seems likely that our methods will also give a simplification of the arguments in \cite{gordon}.  To the best of our knowledge, our special value conjecture is the first such for arithmetic surfaces of characteristic zero, and in this context one of course does not have an analogue of $H^2_{abs}$.

Combined with the analytic class number formula for the Dedekind Zeta function $\zeta_F(s)$ at $s=1$ and $s=0$ formula (\ref{BSDforX}) becomes
\begin{equation} \zeta^*(\X,1)=\frac{2^{r_1}(2\pi)^{r_2}}{\Omega(\X)\sqrt{|D_F|}}\cdot \frac{(\#\Pic^0(\X))_{tor}^2}{\#\Br(\overline{\X})\delta^2(\#\mu_F)^2}\cdot\frac{R_F^2}{R(\X)}\cdot\prod\limits_{\text{$v$ real}}\frac{\delta_v'\delta_v}{\#\Phi_v}.\label{zetaatone}\end{equation}
Since there are now examples of elliptic curves $E/\bq$ for which $\Sha(E)$ is known to be finite and the Birch and Swinnerton-Dyer conjecture is completely proven \cite{Wan}[Thm. 9.3], \cite{LiLiuTian}[Thm. 1.2] our conjectures (\ref{vanishingatone}), (\ref{BSDforX}) and (\ref{zetaatone}) hold for any regular model $\X$ of any principal homogenous space of any elliptic curve isogenous over $\bq$ to such an $E$.

We give a brief summary of each section below. In section \ref{sec:correction} we prove that $C(\X,1)=1$ for general $\X$. Since $\bz(1)=\bg_m[-1]$ the proof involves elementary properties of the sheaf $\bg_m$ in the \'etale topology. However, the $p$-part of $C(\X,1)$ is defined in \cite{Flach-Morin-16}[Def. 5.6]  via the eh-topology on schemes over $\bF_p$ \cite{Geisser06}, and our proof eventually reduces to a curious statement, Prop. \ref{strangeprop} below, which in some sense says that the difference between \'etale and eh-cohomology is the same for the sheaves $\bg_m$ and $\co$. In order to prove Prop. \ref{strangeprop} we develop some results on the eh-topology which might be of independent interest.

In section \ref{sec:av} we prove, for an arithmetic surface $\X$, Conjecture $\mathbf{AV}(\mathcal{X},1)$ of \cite{Flach-Morin-16}[Conj. 6.23] using results of Saito \cite{Saito89}. This conjecture is necessary to define the Weil-\'etale complexes in terms of which our special value conjecture is formulated.

In section \ref{sec:h1} we define Weil-\'etale complexes associated to the relative $H^1$-motive of $f$ and then translate our special value conjecture for $\zeta(H^1,s)$ from the formulation in terms of a fundamental line to the explicit form (\ref{BSDforX}) given above.

Finally, in section \ref{sec:comparison} we prove the equivalence to the Birch and Swinnerton-Dyer conjecture. The two key ingredients are a formula due to Geisser \cite{Geisser17} relating the cardinalities of $\Br(\X)$ and $\Sha(J_F)$ and Lemma \ref{innerprod} about the Arakelov intersection pairing in exact sequences. We give a different proof of Geisser's formula in order to generalize it from totally imaginary to arbitrary number fields $F$.

\section{The correction factor for $n=1$}\label{sec:correction}

We recall some definitions from \cite{Flach-Morin-16}. For any scheme $Z$ and integer $n$ we denote by $\bz(n)=z^n(-,2n-\bullet)$ Bloch's higher Chow complex, viewed as a complex of sheaves on the small \'etale site of $Z$. For any prime number $p$ we set
\begin{equation} R\Gamma(Z,\bz_p(n)):=R\Gamma(Z_{et},\bz_p(n)):=R\varprojlim_\nu R\Gamma(Z_{et},\bz(n)/p^\nu)   \notag\end{equation}
and
\begin{equation} R\Gamma(Z,\bq_p(n)):=R\Gamma(Z,\bz_p(n))\otimes_{\bz_p}\bq_p.\notag\end{equation}

For $\X$ regular, proper over $\Spec(\bz)$ and $n\in\mathbb{Z}$ we consider the derived de Rham complex modulo the Hodge filtration $L\Omega^*_{\mathcal{X}/\mathbb{Z}}/\mathrm{Fil}^n$ (see \cite{Illusie71}[VIII.2.1]) as a complex of abelian sheaves on the Zariski site of $\mathcal{X}$. We obtain a perfect complex of abelian groups
$$R\Gamma_{dR}(\mathcal{X}/\mathbb{Z})/F^n:=R\Gamma(\mathcal{X}_{Zar},L\Omega^*_{\mathcal{X}/\mathbb{Z}}/\mathrm{Fil}^n)$$
and its base change $R\Gamma_{dR}(\mathcal{X}_A/A)/F^n$ to any ring $A$.

Let $(\Schemes^d/\bF_p)_{eh}$ be the category of separated, finite type schemes $Z/\bF_p$ of dimension $\leq d$ with the $eh$-topology \cite{Geisser06}[Def. 2.1] and $(\Sm^d/\bF_p)_{et}$ the full subcategory of smooth schemes with the \'etale topology. One has a natural pair of adjoint functors on abelian sheaves\cite{Geisser06}[Lemma 2.5]
\begin{equation}(\rho_d^*,\rho_{d,*}): \Ab(\Schemes^d/\bF_p)^\sim_{eh} \to \Ab(\Sm^d/\bF_p)_{et}^\sim\label{rhod}\end{equation}
and we set
\begin{equation} R\Gamma(Z_{eh},\bz(n)):=R\Gamma(Z_{eh},\rho_d^*\bz(n)^{SV}),   \notag\end{equation}
\begin{equation} R\Gamma(Z_{eh},\bz_p(n)):=R\varprojlim_\nu R\Gamma(Z_{eh},\rho_d^*\bz(n)^{SV}/p^\nu)   \notag\end{equation}
and
\begin{equation} R\Gamma(Z_{eh},\bq_p(n)):=R\Gamma(Z_{eh},\bz_p(n))\otimes_{\bz_p}\bq_p\notag\end{equation}
where $\bz(n)^{SV}$ is the Suslin Voevodsky motivic complex on the smooth site $\Sm/\bF_p$. For smooth $Z/\bF_p$ there is a quasi-isomorphism
$\bz(n)\cong\bz(n)^{SV}$ on the small \'etale site of $Z$ but $\bz(n)$ is not functorial for arbitrary morphisms in $\Sm/\bF_p$.

Various results on the eh-topology in \cite{Geisser06} are only valid under a resolution of singularities assumption $R(\bF_p,d)$ for varieties of dimension $\leq d$ \cite{Geisser06}[Def. 2.4] and we shall indicate along the way which of our results depend on it. For example, if one assumes $R(\bF_p,d)$ then $\rho_d^*$ is exact. $R(\bF_p,d)$ holds for $d\leq 2$.

Under $R(\bF_p,d)$ and for smooth $Z/k$ one has an isomorphism \cite{Geisser06}[Thm. 4.3]
\begin{equation}  R\Gamma(Z_{eh},\rho_d^*\bz(n)^{SV})\cong R\Gamma(Z_{et},\bz(n)^{SV})\cong R\Gamma(Z_{et},\bz(n)).    \notag\end{equation}

Fix a prime number $p$ and let $\X$ be regular, proper and flat over $\Spec(\bz)$. The following is \cite{Flach-Morin-16}[Conj.5.5].
\begin{conjecture}\label{conjD_p} ${\bf D}_p(\mathcal{X},n)$
There is an exact triangle of complexes of $\mathbb{Q}_p$-vector spaces
$$R\Gamma_{dR}(\mathcal{X}_{\mathbb{Q}_p}/\mathbb{Q}_p)/F^n[-1]\rightarrow R\Gamma(\mathcal{X}_{\mathbb{Z}_p,et},\mathbb{Q}_p(n))\rightarrow R\Gamma(\mathcal{X}_{\mathbb{F}_p,eh},\mathbb{Q}_p(n)).$$
\end{conjecture}

Conjecture \ref{conjD_p} gives an isomorphism $\lambda_p=\lambda_{p}(\mathcal{X},n):$
\begin{align*}
&\left(\mathrm{det}_{\mathbb{Z}_p}  R\Gamma(\mathcal{X}_{\mathbb{Z}_p,et},\mathbb{Z}_p(n))\right)_{\mathbb{Q}_p}\\ \xrightarrow{\sim}&\ \mathrm{det}_{\mathbb{Q}_p}  R\Gamma(\mathcal{X}_{\mathbb{Z}_p,et},\mathbb{Q}_p(n))\\
\xrightarrow{\sim}&\ \mathrm{det}_{\mathbb{Q}_p}  R\Gamma(\mathcal{X}_{\mathbb{F}_p,eh},\mathbb{Q}_p(n))
\otimes_{\mathbb{Q}_p}\mathrm{det}^{-1}_{\mathbb{Q}_p}R\Gamma_{dR}(\mathcal{X}_{\mathbb{Q}_p}/\mathbb{Q}_p)/F^n\\
\xrightarrow{\sim}& \left(\mathrm{det}_{\mathbb{Z}_p}  R\Gamma(\mathcal{X}_{\mathbb{F}_p,eh},\mathbb{Z}_p(n))
\otimes_{\mathbb{Z}_p}\mathrm{det}^{-1}_{\mathbb{Z}_p}R\Gamma_{dR}(\mathcal{X}_{\mathbb{Z}_p}/\mathbb{Z}_p)/F^n\right)_{\mathbb{Q}_p}
\end{align*}
and we define
\begin{align*}d_{p}(\mathcal{X},n)\in \mathbb{Q}_p^\times/\mathbb{Z}_p^\times\end{align*}
such that
\begin{align}&\lambda_{p}\left(d_{p}(\mathcal{X},n)^{-1}\cdot \mathrm{det}_{\mathbb{Z}_p}  R\Gamma(\mathcal{X}_{\mathbb{Z}_p,et},\mathbb{Z}_p(n))\right)\notag\\
=\ &\mathrm{det}_{\mathbb{Z}_p}  R\Gamma(\mathcal{X}_{\mathbb{F}_p,eh},\mathbb{Z}_p(n))
\otimes_{\mathbb{Z}_p}\mathrm{det}^{-1}_{\mathbb{Z}_p}R\Gamma_{dR}(\mathcal{X}_{\mathbb{Z}_p}/\mathbb{Z}_p)/F^n.\notag\end{align}
We then set
\begin{equation}\chi(\mathcal{X}_{\mathbb{F}_p},\mathcal{O},n):= \sum_{i\leq n, j}(-1)^{i+j} \cdot (n-i)\cdot \mathrm{dim}_{\mathbb{F}_p}H^j(\mathcal{X}_{\mathbb{F}_p,eh},\rho_d^*\Omega^i)\label{ehchidef}\end{equation}
where $\Omega^i$ is the sheaf of Kaehler differentials on $\Sm/\bF_p$.  Under $R(\bF_p,d)$ and for smooth $Z/\bF_p$ one has an isomorphism \cite{Geisser06}[Thm. 4.7]
\begin{equation} H^i(Z_{et},\Omega^i)\cong H^i(Z_{eh},\rho_d^*\Omega^i).\notag\end{equation}

We finally set
\begin{equation}c_{p}(\mathcal{X},n):=p^{\chi(\mathcal{X}_{\mathbb{F}_p},\mathcal{O},n)}\cdot d_{p}(\mathcal{X},n)\notag\end{equation}
and
\begin{equation}C(\mathcal{X},n):=\prod_{p<\infty}\mid c_p(\mathcal{X},n)\mid_p.\notag\end{equation}

\begin{prop} Let $\X$ be regular of dimension $d$, proper and flat over $\Spec(\bz)$, and assume $R(\bF_p,d-1)$. Then Conjecture ${\bf D}_p(\mathcal{X},1)$ holds and one has $C(\X,1)=1$.
\label{cgleicheins}\end{prop}

\begin{proof} Fix a prime number $p$. Then the base change $\X_{\bz_p}$ is again a regular scheme. For $n=1$ one has isomorphisms $\bz(1)\cong\bg_m[-1]$ on $\X_{\bz_p,et}$ and $\bz(1)^{SV}\cong\bg_m[-1]$ on $\Sm/\bF_p$.  One has an exact sequence of coherent sheaves
\[ 0\to \mathcal{I}\to \co_{\X_{\bz_p}}\to i_*\co_{\X_{\bF_p}}\to 0\]
and an exact sequence of abelian sheaves
\[ 0\to 1+\mathcal{I}\to \bg_{m,\X_{\bz_p}}\to i_*\bg_{m,\X_{\bF_p}}\to 0\]
on $\X_{\bz_p,et}$. Here $1+\mathcal{I}$ is just our notation for the kernel, the sections of this sheaf over an \'etale $U\to\X_{\bz_p}$ need not coincide with $1+\mathcal{I}(U)$. Passing to $p$-adic completions and shifting by $[-1]$ we obtain an exact triangle
\begin{equation} R\varprojlim_\nu R\Gamma(\X_{\bz_p,et},(1+\mathcal{I})/p^\nu)[-1]\to R\Gamma(\X_{\bz_p,et},\bz_p(1))\to R\Gamma(\X_{\bF_p,et},\bz_p(1)) \to .\label{t1}\end{equation}
For $n=1$ one has an isomorphism $L\Omega^*_{\mathcal{X}/\mathbb{Z}}/\mathrm{Fil}^1\cong \co_\X$ and (\ref{ehchidef}) specializes to
\begin{equation}\chi(\mathcal{X}_{\mathbb{F}_p},\mathcal{O},1)=\chi(\mathcal{X}_{\mathbb{F}_p},\mathcal{O}):=\sum_{j}(-1)^{j} \cdot \mathrm{dim}_{\mathbb{F}_p}H^j(\mathcal{X}_{\mathbb{F}_p,eh},\rho_d^*\co).\label{oeulerchar}\end{equation}
The outer terms in (\ref{t1}) are respectively computed by Lemma \ref{oetale} and Proposition \ref{strangeprop} below. Hence for an arithmetic surface we obtain the exact triangle
$$R\Gamma(\mathcal{X}_{\mathbb{Q}_p},\co_{\mathcal{X}_{\mathbb{Q}_p}})[-1]\rightarrow R\Gamma(\mathcal{X}_{\mathbb{Z}_p,et},\mathbb{Q}_p(1))\rightarrow R\Gamma(\mathcal{X}_{\mathbb{F}_p,eh},\mathbb{Q}_p(1))$$
of Conjecture ${\bf D}_p(\mathcal{X},1)$ after scalar extension to $\bq_p$. Moreover, Lemma \ref{oetale} and Proposition \ref{strangeprop} show that
\begin{align*}\lambda_{p}\left(\mathrm{det}_{\mathbb{Z}_p}  R\Gamma(\mathcal{X}_{\mathbb{Z}_p,et},\mathbb{Z}_p(1))\right)
=\ &p^{\chi_{et}(\mathcal{X}_{\mathbb{F}_p},\mathcal{O})-\chi(\mathcal{X}_{\mathbb{F}_p},\mathcal{O})}\cdot \mathrm{det}_{\mathbb{Z}_p}  R\Gamma(\mathcal{X}_{\mathbb{F}_p,eh},\mathbb{Z}_p(1))\\
&\otimes_{\mathbb{Z}_p}p^{-\chi_{et}(\mathcal{X}_{\mathbb{F}_p},\mathcal{O})}\cdot
\mydet_{\bz_p}^{-1}R\Gamma(\X_{\bz_p},\co_{\X_{\bz_p}})\notag
\end{align*}
inside $\mathrm{det}_{\mathbb{Q}_p}  R\Gamma(\mathcal{X}_{\mathbb{Z}_p,et},\mathbb{Q}_p(1))$, i.e. we have $d_{p}(\mathcal{X},1)=p^{-\chi(\mathcal{X}_{\mathbb{F}_p},\mathcal{O})}$ and therefore $c_{p}(\X,1)=1$.
This finishes the proof of Prop. \ref{cgleicheins}.

\end{proof}

\begin{lemma} There is an isomorphism
\begin{equation} R\varprojlim_\nu R\Gamma(\X_{\bz_p,et},(1+\mathcal{I})/p^\nu)\otimes_{\bz_p}\bq_p\cong R\Gamma(\X_{\bq_p},\co_{\X_{\bq_p}})\label{tangentiso}\end{equation}
so that
\begin{equation} \mydet_{\bz_p}R\varprojlim_\nu R\Gamma(\X_{\bz_p,et},(1+\mathcal{I})/p^\nu)=p^{\chi_{et}(\mathcal{X}_{\mathbb{F}_p},\mathcal{O})}\cdot
\mydet_{\bz_p}R\Gamma(\X_{\bz_p},\co_{\X_{\bz_p}})
\notag\end{equation}
where
\begin{equation}\chi_{et}(\mathcal{X}_{\mathbb{F}_p},\mathcal{O}):= \sum_{j}(-1)^{j} \cdot \mathrm{dim}_{\mathbb{F}_p}H^j(\mathcal{X}_{\mathbb{F}_p,et},\co)\label{oeulercharet}\end{equation}
is the analogue of (\ref{oeulerchar}) for the \'etale (or Zariski) topology.
\label{oetale}\end{lemma}

\begin{proof} Let $(\mathfrak{X},\co_{\mathfrak{X}})$ be the formal completion of $\X_{\bz_p}$ at the ideal $\mathcal{I}=(p)$. The underlying topological space of $\mathfrak{X}$ is $\X_{\bF_p}$ and we denote by $i:\mathfrak{X}\to\X_{\bz_p}$ the natural morphism of ringed spaces and \'etale topoi. We have an exact sequence on $\mathfrak{X}_{et}$
\[ 0\to 1+\mathcal{I}_{\mathfrak{X}}\to \bg_{m,\mathfrak{X}}\to \bg_{m,\X_{\bF_p}}\to 0.\]

\begin{lemma} For any $\nu\geq 1$ the natural morphism $i^*(1+\mathcal{I})\to 1+\mathcal{I}_{\mathfrak{X}}$
induces an isomorphism on $\mathfrak{X}_{et}$
\begin{equation} i^*(1+\mathcal{I})/p^\nu\to (1+\mathcal{I}_{\mathfrak{X}})/p^\nu.\label{formalcomp}\end{equation}
\label{formallemma}\end{lemma}

\begin{proof} Note that (\ref{formalcomp}) is meant in the derived sense, so we have to investigate the kernel and cokernel of multiplication by $p^\nu$ separately. For each (connected) \'etale $\Spec(A)\to\X_{\bz_p}$ the map on kernels is
\begin{equation}\mu_{p^\nu}(A)\cap (1+\mathcal{I})(A)=\mu_{p^\nu}(\hat{A})\cap (1+\mathcal{I}_{\mathfrak{X}})(\hat{A})\label{trivial}\end{equation}
where $\hat{A}$ denotes the $p$-adic completion. Now any element of $\mu_{p^\nu}(A)$ lies in the integral closure of $\bz_p$ in $A$  which coincides with the integral closure of $\bz_p$ in the fraction field $K$ of $A$, as $A$ is regular. Hence $\mu_{p^\nu}(A)$ is contained in the algebraic closure $L$ of $\bq_p$ in $K$. But $L/\bq_p$ is finite as $\X_{\bz_p}\to\Spec(\bz_p)$ is of finite type. We conclude that $\mu_{p^\nu}(A)$ is contained in $\co_L$ which is already $p$-adically complete, i.e. we have $\mu_{p^\nu}(\hat{A})=\mu_{p^\nu}(A)$. Hence the map (\ref{trivial}) is also an isomorphism (in fact both sides vanish unless $p=2$).

Concerning the cokernel of $p^\nu$ we first note that for each \'etale $\Spec(A)\to\X_{\bz_p}$ we have $(1+\mathcal{I}_{\mathfrak{X}})(\hat{A})=1+\mathcal{I}_{\mathfrak{X}}(\hat{A})$ since the inverse of any element $1+p\cdot x$ is given by a geometric series.  Moreover, the usual power series of the exponential and the logarithm induce an isomorphism
\begin{equation}
\log: 1+\mathcal{I}^\nu_{\mathfrak{X}}(\hat{A})=1+p^\nu\cdot\hat{A}\cong p^\nu\cdot\hat{A}=\mathcal{I}^\nu_{\mathfrak{X}}(\hat{A})
\label{logarithm}\end{equation}
for each $\nu\geq 1$ ($\nu\geq 2$ if $p=2$). In particular we conclude that any element in $1+p^{\nu+2}\cdot\hat{A}$ has a $p^\nu$-th root in $1+p\cdot\hat{A}$. Given $1+x\in 1+p\cdot\hat{A}$ and $\nu\geq 1$ we can write
\[ 1+x=1+x_0+p^{\nu+2}\cdot x_1=(1+x_0)\cdot \left(1+p^{\nu+2}\frac{x_1}{1+x_0}\right)\]
with $x_0\in\mathcal{I}(A)=p\cdot A$ and $x_1\in \hat{A}$ and we can also assume that $1+x_0\in A^\times$ since inverting $1+x_0$ does not remove any points of $\X_{\bF_p}=\mathfrak{X}$. Hence we find $1+x\in (1+\mathcal{I})(A)\cdot (1+\mathcal{I}_{\mathfrak{X}}(\hat{A}))^{p^\nu}$, i.e. that (\ref{formalcomp}) is a surjection on cokernels of multiplication by $p^\nu$.

To show that it is also an injection we need to consider an element $1+x_0\in (1+\mathcal{I})(A)$ that becomes a $p^\nu$-th power in $1+\mathcal{I}_{\mathfrak{X}}(\hat{B})$ for some \'etale neighborhood $\Spec(B)\to \Spec(A)$ of any given point $\fp\in\Spec(A/(p))$ and show that $1+x_0\in (1+\mathcal{I}(B'))^{p^\nu}$ in some \'etale neighborhood $\Spec(B')\to \Spec(B)$ of $\fp$.
Since $1+x_0\in (1+\mathcal{I}_{\mathfrak{X}}(\hat{B}))^{p^\nu}$ we have $1+x_0=1+p^{\nu+1}\cdot y_0$. Pick a prime $\fq$ of $B$ above $\fp$. If $y_0\in\fq$ we can write
\[ 1+p^{\nu+1}\cdot y_0=\frac{1+p^{\nu+1}\cdot (y_0+1)}{1+p^{\nu+1}\cdot(1+p^{\nu+1}\cdot y_0)^{-1}}\]
where both $y_0+1$ and $(1+p^{\nu+1}\cdot y_0)^{-1}$ do not lie in $\fq$. Hence we can assume $y_0\notin\fq$ and in fact $y_0\in B^\times$. We then adjoin an element $z$ satisfying the integral equation
\[ f(z):=\frac{(z+p)^{p^\nu}}{p^{\nu+1}}-\frac{z^{p^\nu}}{p^{\nu+1}}-z^{p^\nu}\cdot y_0=\frac{z^{p^\nu}}{p^{\nu+1}}\left[\left(1+\frac{p}{z}\right)^{p^\nu}-(1+p^{\nu+1}\cdot y_0)\right]=0\]
over $B$. Then $B'=B[z,\frac{1}{z}]$ contains a $p^\nu$-th root of $1+p^{\nu+1}\cdot y_0$ and
$$\Spec(B[z,\frac{1}{z}])\to\Spec(B)$$
is \'etale and surjective. Indeed it is clearly \'etale at all primes $\fq'\in\Spec(B)$ with $p\notin\fq'$ and if $p\in\fq'$ we have modulo $\fq'$
\[ f(z)\equiv \begin{cases} z^{p^\nu-1}+(p^\nu-1)z^{p^\nu-2}-z^{p^\nu}\cdot y_0\equiv z^{p^\nu-2}(1+z+z^2\cdot y_0) &\text{if $p=2$}\\z^{p^\nu-1}-z^{p^\nu}\cdot y_0 = z^{p^\nu-1}(1-z\cdot y_0)  & \text{if $p\neq 2$ }\end{cases}\]
so inverting $z$ removes all primes where $\fq'$ is ramified. Moreover, there is a prime $B[z,\frac{1}{z}]\cdot \fq'$ above $\fq'$ with a finite separable residue field extension (in fact the same residue field if $p$ is odd).
\end{proof}

Since $(1+\mathcal{I})/p^\nu$ has torsion cohomology, proper base change gives an isomorphism
\[ R\Gamma(\X_{\bz_p,et},(1+\mathcal{I})/p^\nu)\cong R\Gamma(\X_{\bF_p,et},i^*(1+\mathcal{I})/p^\nu)\]
and Lemma \ref{formallemma} then gives an isomorphism
\begin{equation} R\varprojlim_\nu R\Gamma(\X_{\bz_p,et},(1+\mathcal{I})/p^\nu)\cong R\Gamma(\mathfrak{X}_{et},1+\mathcal{I}_{\mathfrak{X}})\label{formaliso1}\end{equation}
since $1+\mathcal{I}_{\mathfrak{X}}$ is already $p$-adically complete. Consider the diagram with exact rows
\[\begin{CD} R\Gamma(\mathfrak{X}_{et},1+\mathcal{I}^k_{\mathfrak{X}})@>\alpha>> R\Gamma(\mathfrak{X}_{et},1+\mathcal{I}_{\mathfrak{X}})@>>> R\Gamma(\mathfrak{X}_{et},1+\mathcal{I}_{\mathfrak{X}}/1+\mathcal{I}^k_{\mathfrak{X}})@>>>{}\\
@V\log V\sim V @.@.@.\\
R\Gamma(\mathfrak{X}_{et},\mathcal{I}^k_{\mathfrak{X}}) @. @. @.\\
@VV\sim V @.@.@.\\
R\Gamma(\mathfrak{X},\mathcal{I}^k_{\mathfrak{X}})@>\alpha^{add}>>R\Gamma(\mathfrak{X},\co_{\mathfrak{X}})@>>>
R\Gamma(\mathfrak{X},\co_{\mathfrak{X}}/\mathcal{I}^k_{\mathfrak{X}})@>>>{}
\end{CD}\]
where the vertical isomorphism is induced by the logarithm (\ref{logarithm}) for $k$ large enough (in fact $k\geq 2$).
We have isomorphisms of abelian sheaves
\begin{equation}\mathcal{I}^{i}_{\mathfrak{X}}/\mathcal{I}^{i+1}_{\mathfrak{X}}\cong (1+\mathcal{I}^{i}_{\mathfrak{X}})/(1+\mathcal{I}^{i+1}_{\mathfrak{X}});\quad\quad x\mapsto 1+x \label{graded}\end{equation}
and
$R\Gamma(\mathfrak{X},\mathcal{I}^{i}_{\mathfrak{X}}/\mathcal{I}^{i+1}_{\mathfrak{X}})$
is a perfect complex of $\bF_p$-modules since $\mathfrak{X}$ is proper. It follows that the maps $\alpha$ and $\alpha^{add}$ are isomorphisms after tensoring with $\bq_p$. Finally recall the theorem of formal functions \cite{hartshorne}[Thm. III.11.1]
\begin{equation}R\Gamma(\mathfrak{X},\co_{\mathfrak{X}})\cong R\Gamma(\X_{\bz_p},\co_{\X_{\bz_p}}).\label{ff}\end{equation}
The isomorphism (\ref{tangentiso}) is then the composition of the scalar extensions to $\bq_p$ of $(\ref{formaliso1})$, $\alpha$, $\log$, $\alpha^{add}$ and $(\ref{ff})$. The above diagram and (\ref{graded}) show that under this isomorphism we have
\begin{align*} &\mydet_{\bz_p}R\varprojlim_\nu R\Gamma(\X_{\bz_p,et},(1+\mathcal{I})/p^\nu)\\
=&p^{-\chi(1+\mathcal{I}_{\mathfrak{X}}/1+\mathcal{I}^k_{\mathfrak{X}})+\chi(\co_{\mathfrak{X}}/\mathcal{I}^k_{\mathfrak{X}})}\cdot
\mydet_{\bz_p}R\Gamma(\X_{\bz_p},\co_{\X_{\bz_p}})\\
=&p^{\chi(\co_{\mathfrak{X}}/\mathcal{I}_{\mathfrak{X}})}\cdot
\mydet_{\bz_p}R\Gamma(\X_{\bz_p},\co_{\X_{\bz_p}})=p^{\chi_{et}(\mathcal{X}_{\mathbb{F}_p},\mathcal{O})}\cdot
\mydet_{\bz_p}R\Gamma(\X_{\bz_p},\co_{\X_{\bz_p}}).\end{align*}
\end{proof}

\subsection{Results on the eh-topology} To prepare for the proof of Prop. \ref{strangeprop} below we develop some results on the eh-topology which might be of independent interest. We first recall the notion of seminormalization of a scheme.

\begin{definition} The seminormalization $X^{sn}\to X$ of a scheme $X$ is an initial object in the full subcategory of schemes over $X$
consisting of universal homeomorphisms $Z\to X$ which induce isomorphisms on all residue fields.
\end{definition}

By \cite{stacks}[Lemma 28.45.7] the seminormalization always exists and $X^{sn}$ is a {\em seminormal} scheme, meaning that for any affine open $U$ the ring $A=\co(U)$ is a seminormal ring, i.e. if $x^2=y^3$ for some $x,y\in A$ then there is $a\in A$ with $x=a^3$ and $y=a^2$. Any seminormal scheme is reduced \cite{stacks}[Lemma 28.45.5], hence we have a factorization
\[ X^{sn}\to X^{red}\to X\]
by the universal property of either the reduction or the seminormalization.

By\cite{stacks}[28.45.7] the seminormalization is also the final object in the category of seminormal schemes above $X$, i.e.
\begin{equation} \Hom(Z,X^{sn})\xrightarrow{\sim}\Hom(Z,X) \label{univ}\end{equation}
for any seminormal scheme $Z$. Any normal scheme is seminormal and hence the normalization $X^n\to X$ of, say, a Noetherian scheme $X$  \cite{stacks}[Sec. 28.52] factors through the seminormalization $X^n\to X^{sn}$ by (\ref{univ}). If $X$ is moreover a Nagata scheme \cite{stacks}[Def. 27.13.1], for example if $X$ is of finite type over a field, then $X^n\to X$ is a finite morphism by \cite{stacks}[Lemma 28.52.10]. It follows that $X^{sn}\to X$ is finite since $X$ is Noetherian.

\begin{lemma} For any scheme $X$ the seminormalization $X^{sn}\to X$ induces an equivalence of \'etale topoi
\begin{equation}  X^{sn}_{et}\cong X_{et}. \label{snet}\end{equation}
For a scheme $X$ in $\Schemes^d/\bF_p$ the seminormalization $X^{sn}\to X$ induces an isomorphism $$X^{sn,\sim}\cong X^\sim$$
in $(\Schemes^d/\bF_p)^\sim_{eh}$ where $X^\sim$ denotes the eh sheaf associated to the presheaf represented by $X$. Hence for any (abelian) eh-sheaf $\F$ we have
\begin{equation} H^i(X_{eh},\F)\cong H^i(X^{sn}_{eh},\F)\label{sniso}\end{equation}
for all $i$.
\end{lemma}

\begin{proof} The isomorphism (\ref{snet}) follows from the fact that $X^{sn}\to X$ is a universal homeomorphism \cite{sga1}[Exp. IX, 4.10] \cite{sga4}[II, Exp. VIII, 1.1]. Since $X^{sn}\to X$ is finite surjective and induces an isomorphism on residue fields it is an eh-cover by \cite{Geisser06}[Lemma 2.2]. Since it is also a monomorphism
it becomes an isomorphism in $(\Schemes^d/\bF_p)^\sim_{eh}$.
\end{proof}

\begin{lemma} For schemes $X$, $Y$ in $\Schemes^d/\bF_p$ we have
\begin{equation} Y^\sim(X)\cong \Hom_{\Schemes^d/\bF_p}(X^{sn},Y).\notag\end{equation}
\label{ehsheaf}\end{lemma}

\begin{proof} Since $X^{sn,\sim}\cong X^\sim$ we have $Y^\sim(X)=Y^\sim(X^{sn})$ and we can assume that $X$ is seminormal.  We follow the arguments of \cite{Voevodsky96}[3.2]. Since $X$ is reduced, any surjection $X'\to X$ (for example the disjoint union of the schemes in an eh-cover) is an epimorphism in the category of schemes. Hence the separated presheaf associated to $\Hom(-,Y)$ still has value $\Hom(X,Y)$ on seminormal $X$, and the map
\[\Hom(X,Y)\to Y^\sim(X)\]
is injective. Any element $f^\sim$ of the right hand side is represented by a family of morphisms $f_i:U_i\to Y$ on some eh-cover $\{U_i\to X\}_{i\in I}$ with $f_i\vert_{U_i\times_X U_j}=f_j\vert_{U_i\times_X U_j}$. By \cite{Geisser06}[Prop. 2.3] every eh-cover has a refinement of the form $$\{U_i\to X'\to X\}_{i\in I}$$ where $\{U_i\to X'\}_{i\in I}$ is an \'etale cover and $p:X'\to X$ is proper so that for each $x\in X$ there is $y\in p^{-1}(\{x\})$ so that $\kappa(x)\to p_*\kappa(y)$ is an isomorphism. Since $U_i\times_{X'} U_j\to {U_i\times_X U_j}$ we have
$f_i\vert_{U_i\times_{X'} U_j}=f_j\vert_{U_i\times_{X'} U_j}$, i.e. the $f_i$ glue to a morphism $f':X'\to Y$. Consider the Stein factorization $X'\xrightarrow{p_1} X''\to X$ of $p$. Since $p_{1,*}(\co_{X'})=\co_{X''}$ the proof of \cite{Voevodsky96}[Lemma 3.2.7] shows that $f'$ descends to a morphism $f'':X''\to Y$. The proof of \cite{Voevodsky96}[Thm.3.2.9]
produces a factorization $X''\to X'''\xrightarrow{p_0} X$ so that $f''$ descends to $f''':X'''\to Y$ and $p_0$ is a universal homeomorphism. Moreover, $\kappa(p_0^{-1}(x))=\kappa(x)$ for each $x\in X$ since there is a point $y\in X'$ with isomorphic residue field. Since $X$ is seminormal, $p_0$ is an isomorphism and we find that $f^\sim $ is represented by a morphism $f\in\Hom(X,Y)$.

\end{proof}

\begin{corollary} We have
\[\Hom_{(\Schemes^d/\bF_p)^\sim_{eh}}(X^\sim,Y^\sim)=Y^\sim(X)=\Hom(X^{sn},Y^{sn}),\]
i.e. the category of representable sheaves in the eh-topology is equivalent to the category of seminormal schemes.
\end{corollary}

\begin{proof} This follows from Lemma \ref{ehsheaf} together with (\ref{univ}), or by applying Lemma \ref{ehsheaf} to both $Y^{sn}$ and $Y$ together with $Y^{sn,\sim}\cong Y^\sim$.
\end{proof}

Recall the adjunction $(\rho_d^*,\rho_{d,*})$ from (\ref{rhod}).

\begin{lemma} Let $\F$ be a sheaf on $(\Sm^d/\bF_p)_{et}$ representable by a scheme $Y$.  Under $R(d,\bF_p)$ we have $\rho_d^*\F = Y^\sim$ and therefore
\[ \rho_d^*\F(X)=\Hom(X^{sn},Y)\]
for any $X$.
\label{smehsheaf} \end{lemma}

\begin{proof} Since every scheme has a cover by smooth schemes in the eh-topology under $R(d,\bF_p)$, it suffices to show that
$\rho_d^*\F(X) = Y^\sim(X)$ for smooth $X$. Clearly $\rho_d^p\F(X)=\Hom(X,Y)$ where $\rho_d^p$ is the presheaf pullback.
By \cite{Geisser06}[Cor.2.6] every eh-cover of a smooth scheme has a refinement by a cover consisting of smooth schemes. Hence the eh-sheafification process again leads to isomorphic groups on both sides.
\end{proof}

Next we discuss comparison results between \'etale and eh-cohomology. Consider the morphism of topoi $p=p_X:(\Schemes^d/\bF_p)^\sim_{eh}/X\to X_{et}$ and the natural transformation
\begin{equation}\alpha':\F\vert_X\to Rp_*\F\notag\end{equation}
where $\F\vert_X:=p_*\F$ denotes restriction to the small \'etale site. The functor $\F\mapsto \F\vert_X$ extends to complexes but does not preserve quasi-isomorphisms. We obtain a natural transformation
\begin{equation} \alpha:R\Gamma(X_{et},\F\vert_X)\to R\Gamma(X_{eh},\F)  \label{alphatransdef} \end{equation}
on the category of abelian sheaves on $(\Schemes^d/\bF_p)_{eh}$. Both $\alpha'$ and $\alpha$ are also contravariantly functorial in $X$.

\begin{lemma} Assume $\F$ is a torsion sheaf in $\Ab(\Schemes^d/\bF_p)^\sim_{eh}$. Then there exists a natural transformation
\begin{equation} \alpha_c:R\Gamma_c(X_{et},\F\vert_X)\to R\Gamma_c(X_{eh},\F)  \label{alphacdef} \end{equation}
which coincides with (\ref{alphatransdef}) for proper $X$ and is compatible with exact localization triangles for open/closed decompositions $U\hookrightarrow X\hookleftarrow Z$.
\end{lemma}

\begin{proof} Let $j:X\to\bar{X}$ be an open embedding into a proper $\bF_p$-scheme $\bar{X}$ with closed complement $i:Z\hookrightarrow\bar{X}$. Choose an injective resolution $\F\to I^\bullet$ in $\Ab(\Schemes^d/\bF_p)^\sim_{eh}$ and injective resolutions $\F\vert_{\bar{X}}\to J^\bullet_{\bar{X}}$, resp. $i^*(\F\vert_{\bar{X}})\to J^\bullet_Z$, in $\bar{X}_{et}$, resp. $Z_{et}$.
Then the commutative diagram on $\bar{X}_{et}$
\[\begin{CD} \F\vert_{\bar{X}} @>>> i_*i^*(\F\vert_{\bar{X}})\\
@VVV @VVV\\
Rp_{\bar{X},*}\F @>>> i_*Rp_{Z,*}\F
\end{CD}\]
which arises from functoriality of $\alpha'$ for the morphism $i$, is realized by a diagram of maps of complexes of injectives
\[\begin{CD} J^\bullet_{\bar{X}} @>>> i_*J^\bullet_Z\\
@VVV @VVV\\
I^\bullet\vert_{\bar{X}} @>>> i_*(I^\bullet\vert_Z)
\end{CD}\]
with all maps unique up to homotopy and commuting up to homotopy. Note here that $p_*$ and $i_*$ preserve injective objects. Taking global sections on $\bar{X}_{et}$ and taking mapping fibres of the horizontal maps gives $\alpha_c$ unique up to homotopy. For $R\Gamma_c(X_{eh},\F)$ this is the definition of \cite{Geisser06}[Def. 3.3] and for $R\Gamma_c(X_{et},\F\vert_X)$ this is the usual definition of compact support cohomology on the small \'etale site since $\F\vert_X=j^*\F\vert_{\bar{X}}$. As $\F\vert_X$ is torsion this definition is independent of the choice of $j$, and so is $R\Gamma_c(X_{eh},\F)$ by \cite{Geisser06}[Lemma3.4]. The functoriality for open/closed decompositions follows by applications of the octahedral axiom.
\end{proof}

\begin{corollary} Let $\F \in \Ab(\Schemes^d/\bF_p)^\sim_{eh}$ be a torsion sheaf such that
\begin{equation} \alpha:R\Gamma(X_{et},\F\vert_X)\cong R\Gamma(X_{eh},\F) \notag\end{equation}
is an isomorphism for smooth, proper $X$. If one assumes $R(d,\bF_p)$ then
\begin{equation} \alpha_c:R\Gamma_c(X_{et},\F\vert_X)\to R\Gamma_c(X_{eh},\F)  \label{alphacdef1} \end{equation}
is an isomorphism for all $X$.
\label{alphaccor}\end{corollary}

\begin{proof} The proof is a standard induction over the dimension of $X$, see \cite{Geisser06}[Lemma 2.7].
\end{proof}

We remark that
\begin{equation} \alpha:R\Gamma(X_{et},\F\vert_X)\cong R\Gamma(X_{eh},\F)   \label{comp1} \end{equation}
is an isomorphism for all $X$ if $\F\vert_{\Ab(\Schemes^d/\bF_p)^\sim_{et}}$ is constructible, without further assumptions, by \cite{Geisser06}[Thm. 3.6]. The transformation $\alpha_c$ is then also an isomorphism for all $X$ if $\F$ is constructible. However, constructibility is a much stronger assumption than being torsion and will not hold for the $p$-primary torsion sheaves of interest below.

For non-torsion sheaves, even if \'etale and eh cohomology agree on smooth schemes $X$, one cannot expect an isomorphism for general, even normal schemes, as the example \cite{Geisser06}[Prop.8.2] shows. The following Lemma (which is not needed in the remainder of the paper)  allows to prove such an identification between \'etale and eh cohomology in some very restricted circumstances.

\begin{lemma} Let $\F \in \Ab(\Schemes^d/\bF_p)^\sim_{eh}$ be such that
\begin{equation} \alpha:R\Gamma(X_{et},\F\vert_X)\cong R\Gamma(X_{eh},\F)  \notag \end{equation}
is an isomorphism for smooth $X$. Let
\begin{equation}\begin{CD}  Z' @>i'>> X'\\
@VV f' V @VV f V\\
Z @> i >> X
\end{CD}\label{abstractblowup}\end{equation}
be an abstract blowup square \cite{Geisser06}[Def. 2.1] with $f$ finite, and $Z$, $Z'$ and $X'$ smooth. Assume
$f_*\F\vert_{X'}\oplus i_*\F\vert_Z\to g_*\F\vert_{Z'}$
is surjective where $g=if'$.
Then
\[  \alpha:R\Gamma(X_{et},\F\vert_X)\cong R\Gamma(X_{eh},\F)\]
is an isomorphism.
\label{blowup}\end{lemma}

\begin{proof} The functoriality of $p$ in $X$ and diagram (\ref{abstractblowup}) induce a commutative diagram of exact triangles
\[\minCDarrowwidth1em\begin{CD} R\Gamma(X_{et},\F\vert_X) @>>> R\Gamma(X'_{et},\F\vert_{X'})\oplus R\Gamma(Z_{et},\F\vert_{Z}) @>>> R\Gamma(Z'_{et},\F\vert_{Z'}) @>>>{}\\
@VV\alpha_1V @VV\alpha_2V @VV\alpha_3V @.\\
R\Gamma(X_{eh},\F) @>>> R\Gamma(X'_{eh},\F)\oplus R\Gamma(Z_{eh},\F) @>>> R\Gamma(Z'_{eh},\F) @>>>{}\\
\end{CD}\]
where the bottom row is exact by \cite{Geisser06}[Prop. 3.2] and the top row is induced by the exact sequence of sheaves on $X_{et}$
\begin{equation} 0\to \F\vert_X\to f_*\F\vert_{X'}\oplus i_*\F\vert_Z\to g_*\F\vert_{Z'}\to 0.   \notag\end{equation}
Here exactness at the last term was assumed and exactness at the remaining terms follows from the fact that $\F$ is an eh sheaf, and the pullback of (\ref{abstractblowup}) to any \'etale $U\to X$ is again an abstract blowup square. We also use that $f,i,g$ are finite so that $R\Gamma(X_{et},f_*\F\vert_{X'})=R\Gamma(X'_{et},\F\vert_{X'})$ etc. Since $Z',Z,X'$ are smooth the maps $\alpha_2$, $\alpha_3$ are quasi-isomorphisms, and hence so is $\alpha_1$ by the Five Lemma.
\end{proof}

\begin{corollary} Let $X$ be a seminormal curve over $\bF_p$. Then we have isomorphisms
\[  R\Gamma(X_{et},\bg_m)\cong R\Gamma(X_{eh},\rho_d^*\bg_m) \]
and
\[  R\Gamma(X_{et},\co)\cong R\Gamma(X_{eh},\rho_d^*\co). \]
\label{sncurves}\end{corollary}

\begin{proof} Since $R(1,\bF_p)$ holds we have $\rho_d^*\bg_m=\bg_m^\sim$ by Lemma \ref{smehsheaf} and since $X$ is seminormal we have $\bg_m^\sim\vert_X=\bg_m$ by Lemma \ref{ehsheaf}. Consider (\ref{abstractblowup}) with $X'=X^n$, the normalization, and $Z=X^{sing}$, the singular locus. Both $Z$ and $Z'$ are finite unions of closed points, hence smooth, and $X^n$ is smooth since $X$ is a curve. The map $\bg_m\to i'_*\bg_m$ is surjective since $i'$ is a closed embedding, and applying $f_*$ gives a surjection since $f_*$ is exact. So all conditions of
Lemma \ref{blowup} are satisfied. The proof for $\co$ is identical, since $\co$ is representable by $\ba^1$.
\end{proof}

\begin{corollary} Let $X$ be an arbitrary curve over $\bF_p$. Then we have isomorphisms
\[  R\Gamma(X_{eh},\rho_d^*\bg_m)\cong
%R\Gamma(X^{sn}_{eh},\rho_d^*\bg_m)\cong
R\Gamma(X^{sn}_{et},\bg_m) \]
and
\[  R\Gamma(X_{eh},\rho_d^*\co)\cong
%R\Gamma(X^{sn}_{eh},\rho_d^*\co)\cong
R\Gamma(X^{sn}_{et},\co). \]
\label{anycurve}\end{corollary}

\begin{proof} Combine (\ref{sniso}) and Corollary \ref{sncurves}.\end{proof}

We finally come back to the proof of Prop. \ref{cgleicheins}. For $X$ in $\Schemes^d/\bF_p$ we have a natural map
\begin{equation} R\Gamma(X_{et},\bg_m)\to R\Gamma(X_{et},\bg_m^\sim\vert_X) \xrightarrow{\alpha} R\Gamma(X_{eh},\rho_d^*\bg_m)\label{gmres}\end{equation}
by Lemma \ref{smehsheaf}. Consider the induced map on $p$-adic completions and denote by $C$  the [-1]-shift of its mapping cone, so that there is an exact triangle
\[ R\Gamma(X_{et},\bz_p(1))\to R\Gamma(X_{eh},\bz_p(1))\to C.\]
Similarly, we have  an exact triangle
\begin{equation} R\Gamma(X_{et},\co)[-1]\to R\Gamma(X_{eh},\rho_d^*\co)[-1]\to C^{add}.\label{addsequence}\end{equation}
For a bounded complex $K$ with finite cohomology we set
\[\chi(K):=\sum_j(-1)^j\ord_p\#H^j(K).\]
For  example, if $X$ is proper, the terms in (\ref{addsequence}) are perfect complexes of $\bF_p$-vector spaces by \cite{Geisser06}[Cor. 4.8] and we have
\[ \chi(C^{add})=\chi_{et}(X,\mathcal{O})-\chi(X,\mathcal{O})\]
where $\chi(X,\mathcal{O})$, resp. $\chi_{et}(X,\mathcal{O})$, is defined as in (\ref{oeulerchar}), resp. (\ref{oeulercharet}), with $\X_{\bF_p}$ replaced by $X$.

\begin{prop} Let $X\to\Spec(\bF_p)$ be proper of dimension $d$ and assume $R(d,\bF_p)$. Then $C$ is a bounded complex with finite cohomology and
\begin{equation}\chi(C)=\chi(C^{add}).\label{cadd}\end{equation}
In particular for a proper arithmetic scheme $\X$ we have an isomorphism
\[ R\Gamma(\X_{\bF_p,et},\bq_p(1))\cong R\Gamma(\X_{\bF_p,eh},\bq_p(1))\]
so that
\begin{equation}\mathrm{det}_{\mathbb{Z}_p}  R\Gamma(\mathcal{X}_{\mathbb{F}_p,et},\mathbb{Z}_p(1))=
p^{\chi_{et}(\mathcal{X}_{\mathbb{F}_p},\mathcal{O})-\chi(\mathcal{X}_{\mathbb{F}_p},\mathcal{O})}\cdot \mathrm{det}_{\mathbb{Z}_p}  R\Gamma(\mathcal{X}_{\mathbb{F}_p,eh},\mathbb{Z}_p(1))
\notag\end{equation}
inside $\mathrm{det}_{\mathbb{Q}_p}  R\Gamma(\mathcal{X}_{\mathbb{F}_p,eh},\mathbb{Q}_p(1))$.
\label{strangeprop}\end{prop}

\begin{proof} First recall that $\bg_m^\sim\cong\rho_d^*\bg_m$ by Lemma \ref{smehsheaf}. Since $\bg_m^\sim\vert_Z=\bg_m$ for smooth $Z$, the assumption of Cor. \ref{alphaccor} is satisfied for $\bg_m/p^\nu$ by \cite{Geisser06}[Thm. 4.3], and we deduce that the $p$-adic completion of $\alpha$ in (\ref{gmres}) is an isomorphism for proper $X$. Moreover, by Lemma \ref{ehsheaf} we have $\bg_m^\sim\vert_X=\sigma_*\bg_m$ where $\sigma:X^{sn}\to X$ is the seminormalization. So we obtain an exact triangle
\[ R\Gamma(X_{et},\bz_p(1))\to R\Gamma(X^{sn}_{et},\bz_p(1))\to C.\]
Since $\co$ is a $p$-torsion sheaf an analogous argument gives an exact triangle
\[ R\Gamma(X_{et},\co)[-1]\to R\Gamma(X^{sn}_{et},\co)[-1]\to C^{add}.\]
We have an exact sequence of coherent sheaves on $X_{et}\cong X^{sn}_{et}$
\[ 0\to J\to \co_X\to \sigma_*\co_{X^{sn}}\to K \to 0\]
where $J$ is the nilradical, as $X^{sn}$ is reduced. There is an analogous sequence of abelian sheaves
\[ 1\to 1+J\to \bg_{m,X}\to \sigma_*\bg_{m,X^{sn}}\to K^{mult}\to 1.\]
Both $J$ and $1+J$ have finite filtrations $J^k$, resp. $1+J^k$, with isomorphic subquotients
\begin{equation}J^{k}/J^{k+1}\cong (1+J^{k})/(1+J^{k+1});\quad\quad x\mapsto 1+x, \notag\end{equation}
hence
\begin{equation}\chi(R\Gamma(X_{et},J))=\chi(R\Gamma(X_{et},1+J)).\label{nil}\end{equation}
To prove (\ref{cadd}) it then suffices to show $\chi(R\Gamma(X_{et},K))=\chi(R\Gamma(X_{et},K^{mult}))$. This in turn will follow from the more general statement that
\begin{equation}\chi(R\Gamma(X_{et},\sigma'_*\co_{X'}/\co_{X}))=\chi(R\Gamma(X_{et},\sigma'_*\bg_{m,X'}/\bg_{m,X}))\label{sigmaprime}\end{equation}
for any universal homeomorphism $$\sigma':X'\to X$$ inducing isomorphisms on all residue fields with $X$ reduced. We prove this by induction on the dimension $d$ of $X$. First we can assume that $X'$ is reduced, using (\ref{nil}) for the nilradical of $X'$. If $d=0$ both $X$ and $X'$ are finite unions of spectra of finite fields and $\sigma'$ is an isomorphism. In general, let $Z\hookrightarrow X$ be the singular locus, a proper closed subset of $X$, with its reduced scheme structure, and let
\[\begin{CD} Z' @>>> X'\\ @VVV @VV\sigma' V\\ Z @>>> X\end{CD}\]
be the pullback under $\sigma'$. Then $Z'\to Z$ is again a universal homeomorphism inducing isomorphisms on all residue fields with $Z$ reduced, to which the induction assumption applies.

Since $X\setminus Z$ is smooth, hence seminormal, the restriction of $\sigma'$ to $X'\setminus Z'$ has a section, hence is an isomorphism as $X'$ is reduced.  It follows that $\co_{X'}/\co_X$ is supported on $Z$ (we omit $\sigma'_*$ since $\sigma'$ is a homeomorphism). The morphism $\sigma'$ is finite as the seminormalization factors through it. Hence $\co_{X'}/\co_X$ is coherent and there exists $r$ such that $\cm^r \cdot \co_{X'}/\co_X=0$ where $\cm$ denotes the ideal sheaf of $Z$. Setting $\cm'=\cm\co_{X'}$ we have an exact sequence of coherent sheaves
\[ 0\to \cm'/\cm\to \co_{X'}/\co_X\to  \co_{Z'}/\co_Z\to 0\]
supported on $Z$. Since $(\cm')^{r+1}=\cm^{r+1}\co_{X'}\subseteq\cm$ we have a finite filtration by coherent subsheaves
\[ \cm'/\cm \supseteq \cdots\supseteq \mathcal{N}^i \subseteq\cdots\supseteq  \mathcal{N}^{r+1}=0\]
where $\mathcal{N}^i:=(\cm')^i/(\cm')^i\cap\cm$.
On the multiplicative side we have an exact sequence of abelian sheaves on $X_{et}$ supported in $Z$
\[ 0\to (1+\cm')/(1+\cm)\to \bg_{m,X'}/\bg_{m,X}\to  \bg_{m,Z'}/\bg_{m,Z}\to 1\]
with corresponding filtration
\[ (1+\cm')/(1+\cm)=1+\cm'/\cm \supseteq \cdots\supseteq 1+\mathcal{N}^i \supseteq \cdots\supseteq 1+\mathcal{N}^{r+1}=1\]
and an isomorphism of subquotients.
\[ \mathcal{N}^i/\mathcal{N}^{i+1}\cong (1+\mathcal{N}^i)/(1+\mathcal{N}^{i+1});\quad  x\mapsto 1+x.\]
Note here that all sections $1+x\in 1+\cm'/\cm$ are invertible, since $x^{r+1}\in\cm$. We conclude that
\begin{equation}\chi(R\Gamma(X_{et},\cm'/\cm))=\chi(R\Gamma(X_{et},1+\cm'/\cm))\notag\end{equation}
and together with the induction assumption we obtain (\ref{sigmaprime}).
\end{proof}

\section{Artin-Verdier duality}\label{sec:av}

A key ingredient in our construction of Weil-\'etale complexes is Artin-Verdier duality with torsion coefficients, in the form of Conjecture $\mathbf{AV}(\mathcal{X},n)$ introduced in \cite{Flach-Morin-16}[Conj. 6.23]. The compact support cohomology $\hat{H}^i_c$ is defined as in \cite{milduality}[II, 2.3] using Tate cohomology at all archimedean places.

\begin{conjecture}$\mathbf{AV}(\mathcal{X},n)$ There is a symmetric product map
$$\mathbb{Z}(n)\otimes^L \mathbb{Z}(d-n)\rightarrow \mathbb{Z}(d)$$ in $\mathcal{D}(\mathcal{X}_{et})$
such that the induced pairing
$$\widehat{H}^{i}_c(\mathcal{X}_{et},\mathbb{Z}/m(n))\times H^{2d+1-i}(\mathcal{X}_{et},\mathbb{Z}/m(d-n))\rightarrow \widehat{H}^{2d+1}_c(\mathcal{X}_{et},\mathbb{Z}/m(d))\rightarrow\mathbb{Q}/\mathbb{Z}$$
is a perfect pairing of finite abelian groups for any $i\in\mathbb{Z}$ and any positive integer $m$.
\label{av}\end{conjecture}

This conjecture is known for $\mathcal{X}$ smooth proper over a number ring, and for regular proper $\mathcal{X}$ as long as $n\leq 0$ or $n\geq d$. Therefore, if $\X$ is an arithmetic surface, the only unresolved case is $n=1$ which we shall prove using results of Saito in \cite{Saito89}.

\begin{prop} Conjecture $\mathbf{AV}(\mathcal{X},1)$ holds if $\X$ is an arithmetic surface.
\label{avprop}\end{prop}

\begin{proof} It suffices to prove the statement for an arbitrary prime power  $m=p^\nu$. Consider the open/closed decomposition
\begin{equation}\begin{CD} \X_{\bF_p} @>i>> \X @<j<< \X[1/p]\\
@VVV @VVfV @VV\tilde{f}V\\
S_{\bF_p} @>>> S @<<< S[1/p]
\end{CD}\label{openclosed}\end{equation}
and note that $j^*\bz(1)/p^\nu\cong \mu_{p^\nu}$. Denoting by $A^*:=R\Hom(A,\bq/\bz)$ the $\bq/\bz$-dual, we have isomorphisms
\begin{align}R\Gamma(\X[1/p],\mu_{p^\nu})\cong & R\Gamma(\X[1/p],R\underline{\Hom}_{\X_[1/p]}(\mu_{p^\nu},\mu_{p^\nu}^{\otimes 2}))\notag\\
\cong & R\Hom_{\X[1/p]}(\mu_{p^\nu},\mu_{p^\nu}^{\otimes 2})\notag\\
\cong & R\Hom_{\X[1/p]}(\mu_{p^\nu},R\tilde{f}^!\mu_{p^\nu})[-2]\notag\\
\cong & R\Hom_{S[1/p]}(R\tilde{f}_*\mu_{p^\nu},\mu_{p^\nu})[-2]\notag\\
\cong &\hat{R}\Gamma_c(S[1/p],R\tilde{f}_*\mu_{p^\nu})^*[-5] \notag\\
\cong &\hat{R}\Gamma_c(\X[1/p],\mu_{p^\nu})^*[-5]\label{openduality}
\end{align}
using purity $\tilde{f}^!\mu_{p^\nu}\cong \mu_{p^\nu}^{\otimes 2}[2]$, as $\X$ and $S$ are regular of dimensions $2$ and $1$, respectively \cite{fuji96}, the adjunction $(R\tilde{f}_*,\tilde{f}^!)$ since $\tilde{f}$ is proper, and Artin-Verdier duality on $S[1/p]$ \cite{milduality}[II, Thm. 3.1]. Also note that $R\tilde{f}_*\mu_{p^\nu}$ is a complex with constructible cohomology as $\tilde{f}$ is proper.

Let $\mathfrak{X}$ be the formal completion of $\X$ at the ideal sheaf $I=(p)$ where we view the structure sheaf $\co_{\mathfrak{X}}$ as a sheaf on $\X_{\bF_p,et}$. By \cite{Saito89}[Thm. 4.13] the map
\begin{equation} H^{4-i}_{\X_{\bF_p}}(\X,\bg_m)\to H^i(\X_{\bF_p},\bg_{m,\mathfrak{X}})^*\label{saitomap}\end{equation}
is an isomorphism for $i=0,1$. The groups $H^i(\X_{\bF_p},\bg_{m,\mathfrak{X}})$ can be computed following the computation of
$H^i(\X_{\bF_p},i^*\bg_m)$ in \cite{Saito89}[Prop. 4.6 (2)], and the result is the same for $i\neq 0,1$. More precisely, the rings $\co_x$, $\co_\eta$, $\co_\lambda$ of loc. cit. get replaced by the corresponding local rings of $\mathfrak{X}$, which are again Henselian local with unchanged residue field. Since $\bg_m$ is a smooth group scheme, its cohomology in degrees $\geq 1$ coincides with that of the residue field \cite{miletale}[III. 3.11].

It follows then from \cite{Saito89}[Prop. 4.6 (1),(2)] that (\ref{saitomap}) is an isomorphism for $i\neq 3$ and has cokernel the uniquely divisible group $(\hat{\bz}/\bz)^r$ for $i=3$ where $r$ is the number of irreducible components of $\X_{\bF_p}$. Hence
\begin{equation} H^{3-i}_{\X_{\bF_p}}(\X,\bg_m/p^\nu)\to H^i(\X_{\bF_p},\bg_{m,\mathfrak{X}}/p^\nu)^*\notag\end{equation}
is an isomorphism for all $i$. On the other hand by Lemma \ref{formallemma} we have an isomorphism $i^*\bg_m/p^\nu\cong\bg_{m,\mathfrak{X}}/p^\nu$ and hence an isomorphism
\begin{equation} H^{3-i}_{\X_{\bF_p}}(\X,\bg_m/p^\nu)\xrightarrow{\sim} H^i(\X_{\bF_p},i^*\bg_m/p^\nu)^*\notag\end{equation}
or equivalently
\begin{equation} H^{5-i}_{\X_{\bF_p}}(\X,\bz(1)/p^\nu)\xrightarrow{\sim} H^i(\X_{\bF_p},i^*\bz(1)/p^\nu)^*\label{saitomap2}\end{equation}
for all $i$. Now consider the duality map on localization triangles
\[\minCDarrowwidth1em\begin{CD}
R\Gamma_{\X_{\bF_p}}(\X,\bz(1)/p^\nu) @>>>R\Gamma(\X,\bz(1)/p^\nu) @>>> R\Gamma(\X[1/p],\mu_{p^\nu}) @>>>\\
@V(\ref{saitomap2})VV  @VVV @V (\ref{openduality})VV @.\\
R\Gamma(\X_{\bF_p},i^*\bz(1)/p^\nu)^*[-5]@>>> \hat{R}\Gamma_c(\X,\bz(1)/p^\nu)^*[-5] @>>> \hat{R}\Gamma_c(\X[1/p],\mu_{p^\nu})^*[-5] @>>>{}
\end{CD}\]
where we have shown the outer vertical maps to be quasi-isomorphisms. It follows that the middle vertical map is a quasi-isomorphism which finishes the proof of Prop. \ref{avprop}.
\end{proof}

The remainder of this section is aimed at the proof of Corollary \ref{compatibleduality} below. The results of Saito involve Lichtenbaum's pairing \cite{Saito89}[(1.1)]
\begin{equation} \bg_m[-1]\otimes^L\bg_m[-1]\to\bz(2)^{Li}\label{saitopairing}\end{equation}
together with the trace map
\begin{equation}H^6(\X,\bz(2)^{Li})\to\bq/\bz\notag\end{equation}
constructed in \cite{Saito89}[Thm. 3.1]. The relation to Bloch's complex $\bz(2)$ is given by maps
\begin{equation}  \bz(2)^{Li}\to \mathcal K_{/\X}\xleftarrow{\sim} \tau^{\geq 1}\bz(2)\leftarrow\bz(2)   \label{zhongmap}\end{equation}
where $\mathcal K_{/\X}$ is the complex constructed by Spiess in \cite{Spiess96} and the middle isomorphism is \cite{Zhong14}[Thm. 3.8].
The pairing (\ref{saitopairing}), combined with the map (\ref{zhongmap}) and taken modulo $p^\nu$, can also be characterized as the unique pairing constructed by the method of Sato in \cite{Sato07}. Even though Sato assumes $\X$ to have semistable reduction, his arguments work in our situation where $\X$ is a relative curve and $n=m=1$. We summarize the properties we need in the following Proposition.

\begin{prop} a) The map $(\tau^{\geq 1}\bz(2))/p^\nu\leftarrow\bz(2)/p^\nu$ is a quasi-isomorphism.

b) The pairing
\begin{equation} \bz(1)/p^\nu\otimes^L\bz(1)/p^\nu\to \bz(2)/p^\nu\label{satopairing}\end{equation}
obtained by combining (\ref{saitopairing}), (\ref{zhongmap}) and a) is the unique pairing extending $$\mu_{p^\nu}\otimes\mu_{p^\nu}\cong \mu_{p^\nu}^{\otimes 2}$$ on $\X[1/p]_{et}$.

c) There is a commutative diagram in $D(S_{et},\bz/p^\nu)$
\begin{equation}
\begin{CD} Rf_*\bz(1)/p^\nu[2]\otimes^L \bz(1)/p^\nu @>\trace^0_f\otimes\id>>\bz(0)/p^\nu\otimes^L \bz(1)/p^\nu@=\bz(1)/p^\nu\\
@V\id\otimes f^*VV @.\Vert\\
Rf_*\bz(1)/p^\nu[2]\otimes^L Rf_*\bz(1)/p^\nu @>>>Rf_*\bz(2)[2]/p^\nu@>\trace^1_f>> \bz(1)/p^\nu
\end{CD}
\label{satodiagram}\end{equation}
where $\trace_f$, resp. $f^*$, is constructed as in \cite{Sato07}[Thm. 7.1.1], resp.  \cite{Sato07}[Prop. 4.2.8].
\label{satoprop}\end{prop}

\begin{proof} Since $\mathcal K_{/\X}$ is concentrated in degrees $1$ and $2$ by \cite{Spiess96}[1.6.2.(A1)], it follows that $\bz(2)$ is concentrated in degrees $\leq 2$, i.e. \cite{Flach-Morin-16}[Conj. 7.1] holds true. By \cite{Flach-Morin-16}[Lemma 7.7] there is then an exact triangle
\begin{equation} \tau^{\leq 1}(i_*\bz(1)^Z/p^\nu)[-2]\to \bz(2)/p^\nu \to \tau^{\leq 2} Rj_*\mu_{p^\nu}^{\otimes 2}\to\label{localization}\end{equation}
where $i$ and $j$ are as in (\ref{openclosed}) and we set $Z:=\X_{\bF_p}$. Hence $\bz(2)/p^\nu$ is concentrated in degrees $0,1,2$ and $\mathcal H^0(\bz(2)/p^\nu)\cong j_*\mu_{p^\nu}^{\otimes 2}\cong j_*\mathcal H^1(\bz(2))_{p^\nu}$. This last identity follows from
$$(\mathcal H^1\mathcal K_{/\X})_{p^\nu}\cong \eta_*\mathcal H^1(\bz(2)^{Li})_{p^\nu}\cong \eta_*K_3^{ind}(\bar{L})_{p^\nu}\cong \eta_*\mu_{p^\nu}^{\otimes 2}$$ where $\eta:\Spec(L)\to\X$ is the function field of $\X$ (see the proof of \cite{Spiess96}[1.5.1]). This concludes the proof of a).

By \cite{Zhong14}[Thm.1.1] the complex $\bz(1)^Z/p^\nu$ is quasi-isomorphic to the logarithmic deRham-Witt complex
\begin{align} \bz(1)^Z/p^\nu\cong&\left[ \bigoplus_{z\in Z^0}i_{z,*}W_\nu\Omega^1_{\kappa(z),\log}\xrightarrow{\partial}\bigoplus_{x\in Z^1}i_{x,*}W_\nu\Omega^0_{\kappa(x),\log} \right] \notag\\
\cong &\left[ \bigoplus_{z\in Z^0}i_{z,*}(\kappa(z)^\times)/p^\nu\xrightarrow{\partial}\bigoplus_{x\in Z^1}i_{x,*}\bz/p^\nu                       \right]
\notag\label{logdr}\end{align}
placed in degrees $1$ and $2$. Note, incidentally, that this complex only depends on the underlying topological space and the residue fields of $Z$, hence coincides for $Z$ and $Z^{sn}$. The stalk of $\partial$ at a point $x\in Z^1$ is the map
\[ \bigoplus_{x\in\overline{\{z\}}}(\kappa(z)^\times)/p^\nu \xrightarrow{\sum_{z,P} v_P} \bz/p^\nu\]
where for each irreducible component $\overline{\{z\}}$ of $Z$ passing through $x$ the sum is over all valuations of $\kappa(z)$ lying over $x$, i.e. over all points $P$ of the normalization of $\overline{\{z\}}$ lying above $x$. Since we are considering stalks in the \'etale topology, we can assume the base field is $\bar{\bF}_p$, hence infinite. Let $U\subseteq \ba^N$ be an affine neighborhood of the points $P$. Then there exists a hyperplane $H\subseteq\ba^N$ intersecting $U$ transversely in one of the $P$, say $P_0$, and not in any other $P$. The linear form with zero set $H$, restricted to $U$, gives a function $f\in\kappa(z)$ with $v_{P_0}(f)=1$ and $v_{P}(f)=0$ for $P\neq P_0$, i.e. $\partial_x(f)=1$. Hence $\partial$ is surjective and
\[\bz(1)^Z/p^\nu\cong \ker(\partial)[-1]\]
is concentrated in degree one. In the notation of \cite{Sato07}[2.2] one has $\ker(\partial)\cong \nu^1_{Z,\nu}$. The localisation triangle
(\ref{localization}) then shows that $\bz(2)/p^\nu$ satisfies the defining properties of the complex $\mathfrak T_\nu(2)_\X$ constructed by Sato in \cite{Sato07}[Lemma 4.2.2] (under the semistable reduction assumption). Similarly, the complex $\bz(1)/p^\nu=\bg_m[-1]/p^\nu$ satisfies the defining properties of $\mathfrak T_\nu(1)_\X$ since the proof of \cite{Sato07}[Prop. 4.5.1] in fact goes through for arbitrary regular $\X$.

The proof of \cite{Sato07}[Prop. 4.2.6] then goes through to characterize (\ref{satopairing}) as the unique product extending $\mu_{p^\nu}\otimes\mu_{p^\nu}\cong \mu_{p^\nu}^{\otimes 2}$. This gives b).

Finally, the proof of \cite{Sato07}[Cor. 7.2.4] for $f:\X\to S$, hence $c=-1$, $n=1$, $m=0$ goes through to give c).

\end{proof}

\begin{corollary} There is a commutative diagram of duality isomorphisms
\[\begin{CD}
R\Gamma(S,\bz(1)/p^\nu) @>>> R\Gamma(\X,\bz(1)/p^\nu) \\
@V \mathbf{AV}(S,1) VV  @V\mathbf{AV}(\mathcal{X},1) VV \\
\hat{R}\Gamma_c(S,\bz/p^\nu)^*[-3] @>>> \hat{R}\Gamma_c(\X,\bz(1)/p^\nu)^*[-5]
\end{CD}\]
where the top, resp. bottom, horizontal map is induced by $f^*$, resp. $\trace_f^0$ in (\ref{satodiagram}).
\label{compatibleduality} \end{corollary}

\section{Isolating the $H^1$-part}\label{sec:h1}

The formulation of the special value conjecture \cite{Flach-Morin-16}[Conj 5.12] for $\zeta(\X,s)$ at $s=1$ involves the fundamental line
\[ \Delta(\X/\bz,1):={\det}_\bz R\Gamma_{W,c}(\mathcal{X},\mathbb{Z}(1))\otimes_\bz {\det}_\bz R\Gamma(\X_{Zar},L\Omega_{\X/\bz}^{<1})\]
and the exact triangle \cite{Flach-Morin-16}[(5)]
\begin{equation} R\Gamma(\X_{Zar},L\Omega_{\X/\bz}^{<1})_\br[-2]\to R\Gamma_{\Ar,c}(\X,\tr(1))\to R\Gamma_{W,c}(\mathcal{X},\mathbb{Z}(1))_\br\to \label{tangenttri}\end{equation}
which induces the trivialization (\ref{ourtheta}) of the determinant of $\Delta(\X/\bz,1)_\br$.
For each complex in (\ref{tangenttri}), we shall define in this section a corresponding complex for the relative $H^1$-motive of the morphism $f:\X\to S$ and obtain a corresponding description of $\zeta(H^1,s)$ at $s=1$. In the absence of a suitable triangulated category of motivic complexes $DM$ with a motivic $t$-structure, we isolate the relative $H^1$-motive in an ad-hoc way in the derived category of \'etale sheaves on $S$. More precisely, by \cite{grothbrauer}[Cor. 3.2] one has $R^if_*\bg_m=0$ for $i\geq 2$ and
\[ P=\Pic_{\X/S}:=R^1f_*\bg_m\]
is the relative Picard functor studied for example in \cite{ray}. One has a truncation triangle
\begin{equation} \bg_m \to Rf_*\bg_m\to P[-1]\to \label{trunc1}   \end{equation}
and we define a complex of \'etale sheaves $P^0$ by the exact triangle
\begin{equation}  P^0\to P \xrightarrow{\deg}\bz\to    \label{trunc2} \end{equation}
where the degree map is discussed for example in \cite{grothbrauer}[Sec. 4]. The complex $P^0$ serves as a substitute for the relative $H^1$-motive and we will define Weil-\'etale and Weil-Arakelov complexes associated to it according to the following table. The first column refers to definitions made in \cite{Flach-Morin-16}. In particular, $\overline{\X}$ and $\overline{S}$ denote the Artin-Verdier compactifications of $\X$ and $S$, respectively. If for example $f$ has a section the complexes in the right hand column are direct summands of those in the left hand column.
In general the exact triangles (\ref{trunc1}) and (\ref{trunc2}) will induce corresponding exact triangles for the Weil-\'etale complexes associated to $\bz(1)$, $\bz$, $P$ and $P^0$ on $S$.
\[
\begin{array}{c|c}
R\Gamma_W(\overline{\mathcal{X}},\mathbb{Z}(1))&R\Gamma_W(\overline{S},P^0)[-2]\\\hline\\
R\Gamma_{W,c}(\mathcal{X},\mathbb{Z}(1)) & R\Gamma_{W,c}(S,P^0)[-2]\\\hline\\
R\Gamma_W(\mathcal{X}_\infty,\mathbb{Z}(1))&R\Gamma_W(S_\infty,P^0)[-2]\\\hline\\
R\Gamma_{\Ar,c}(\X,\tr(1)) & \Pic^0(\X)_\br[-2]\oplus\Pic^0(\X)_\br[-3]\\\hline\\
R\Gamma(\X_{Zar},L\Omega_{\X/\bz}^{<1}) & H^1(\X,\co_{\X})[-1]\\
\end{array}
\]

The precise definition of Weil-\'etale modifications will be recalled in the proof of the following Lemma.

\begin{lemma} If $\Br(\X)$ is finite then $R\Gamma_W(\overline{S},P^0)$ is a perfect complex of abelian groups satisfying a duality
$$R\Gamma_W(\overline{S},P^0)\xrightarrow{\sim}  R\mathrm{Hom}_\bz(R\Gamma_W(\overline{S},P^0),\mathbb{Z}[-1]).$$
Its cohomology is given by
\[ H^i_W(\overline{S},P^0)=\begin{cases} \Pic^0(\X)/\Pic(\co_F) & i=0 \\ H^1(\overline{S},P^0)\oplus\Hom_\bz(\Pic^0(\X),\bz) & i=1\\
(\Pic^0(\X)_{tor}/\Pic(\co_F))^* & i=2 \end{cases}\]
where $H^1(\overline{S},P^0)=H^1(\overline{S}_{et},P^0)$ is a finite abelian group of cardinality $\#\Br(\overline{\X})\delta^2$ and $\Br(\overline{\X})$ is defined by the exact sequence
\begin{equation} 0\to\Br(\overline{\X})\to\Br(\X)\to \bigoplus\limits_{\text{$v$ real}}\Br(\X_{F_v}).\label{brauerdef}\end{equation}
In particular, $\Br(\overline{\X})$ coincides with $\Br(\X)$ if $F$ has no real places and is a subgroup of $\Br(\X)$ of co-exponent $2$ in general.
%The duality \cite{Flach-Morin-16}[Thm. 3.22]
%$$R\Gamma_W(\overline{\mathcal{X}},\mathbb{Z}(1))\xrightarrow{\sim}  R\mathrm{Hom}(R\Gamma_W(\overline{\mathcal{X}},\mathbb{Z}(1)),\mathbb{Z}[-5])$$
%together with the corresponding duality on $S$
%$$R\Gamma_W(\overline{S},\mathbb{Z}(1))\xrightarrow{\sim}  R\mathrm{Hom}(R\Gamma_W(\overline{S},\mathbb{Z}),\mathbb{Z}[-3])$$
%induces a duality
\label{h1br}\end{lemma}

\begin{proof}  According to their definition in \cite{Flach-Morin-16}[App. A] the Artin-Verdier \'etale topoi of $\X$ and $S$ fit into a commutative diagram of morphisms of topoi
\begin{equation}\begin{CD} \X_{et} @>\phi>> \overline{\X}_{et} @<u_\infty << \X_\infty @<\pi << Sh(G_\br,\X(\bc))\\
@VfVV @V\overline{f}VV @Vf_\infty VV @Vf_\bc VV\\
S_{et} @>\phi^S>> \overline{S}_{et} @<u^S_\infty << S_\infty @<\pi^S << Sh(G_\br,S(\bc)).
\end{CD}\label{avdef}\end{equation}
Applying $R\overline{f}_*$ to the defining exact triangle \cite{Flach-Morin-16}[App. A, Cor. 6.8] of the complex $\bz(1)^{\overline{\X}}$ we obtain a commutative diagram of exact triangles
\begin{equation}\begin{CD} \bz(1)^{\overline{S}}@>>> R\phi^S_*\bz(1) @>>> u^S_{\infty,*}\tau^{>1}R\hat{\pi}^S_*(2\pi i\bz)\\
@VVV @VVV @VVV\\
R\overline{f}_*\bz(1)^{\overline{\X}}@>>> R\phi^S_*Rf_*\bz(1) @>>> u^S_{\infty,*}Rf_{\infty,*}\tau^{>1}R\hat{\pi}_*(2\pi i\bz)\\
@VVV @VVV @VVV\\
P^{\overline{S}}[-2]@>>> R\phi^S_*P[-2]@>>> u^S_{\infty,*}K[-2]
\end{CD}\label{sbarcomp}\end{equation}
where the middle vertical triangle is induced by (\ref{trunc1}) and the left, resp. right, vertical triangle defines $P^{\overline{S}}$, resp. $K$. Since $R\hat{\pi}_*(2\pi i\bz)$, resp. $R\hat{\pi}^S_*(2\pi i\bz)$, is a complex of sheaves supported on $\X(\br)$, resp. $S(\br)$, with stalk $H^j(G_\br,2\pi i\bz)=\bz/2\bz$ in odd degrees $j$, and vanishing in even degrees, it follows that the top two complexes in the right hand column are supported in degrees $\geq 3$. This gives exactness of the columns in the diagram
\[\minCDarrowwidth1em\begin{CD} {} @. {}@. {}@.{}@. 0 @. 0\\
@. @. @. @. @VVV @VVV\\
0 @>>> H^2(\overline{S},\bz(1)) @>>> H^2(\overline{\X},\bz(1))@>>> H^0(\overline{S},P)@>>> H^3(\overline{S},\bz(1)) @>>>  H^3(\overline{\X},\bz(1))\\
@. @V\cong VV @V\cong VV @VVV @VVV @VVV\\
 0 @>>> \Pic(\co_F) @>>> \Pic(\X) @>>> H^0(S,P)@>>>\Br(\co_F) @>>> \Br(\X)\\
@. @. @. @. @VVV @VVV\\
  {}@.{}@.{}@.{}@. \bigoplus\limits_{\text{$v$ real}}\Br(F_v) @>>> \bigoplus\limits_{\text{$v$ real}}\Br(\X_{F_v})
\end{CD}\]
which is the map of long exact cohomology sequences induced by the left two columns in (\ref{sbarcomp}). Here we use the isomorphism (see the proof of Lemma \ref{pbar} below)
\[ H^3(\X_\infty, \tau^{>1}R\hat{\pi}_*(2\pi i\bz))\cong H^3(G_\br,\X(\bc),(2\pi i)\bz)\cong H^3(\X_{\br,et},\bz(1))\cong \bigoplus\limits_{\text{$v$ real}}\Br(\X_{F_v})\]
and similarly for $S_\infty$. We deduce that $\Br(\overline{S}):=H^3(\overline{S},\bz(1))=0$, hence
\begin{equation} H^0(\overline{S},P)\cong \Pic(\X)/\Pic(\co_F)\label{h0comp}\end{equation}
and that $H^3(\overline{\X},\bz(1))$ coincides with the group $\Br(\overline{\X})$ defined in (\ref{brauerdef}). The continuation of the top long exact sequence gives
\[\minCDarrowwidth1em\begin{CD} 0 @>>> H^3(\overline{\X},\bz(1))@>>> H^1(\overline{S},P)@>>> H^4(\overline{S},\bz(1))@>>> H^4(\overline{\X},\bz(1))\\
@. @. @. @V\cong VV @V\cong VV\\
{}@. {}@.  @. H^0(\overline{S},\bz)^* @>\deg^*>> H^2(\overline{\X},\bz(1))^*
\end{CD}\]
where the vertical duality isomorphisms follow from \cite{Flach-Morin-16}[Prop. 3.4] for both $S$ and $\X$, taking into account the compatibility of dualities in Corollary \ref{compatibleduality}. Since we have a commutative diagram
\[\begin{CD} \Pic(\X) @>\deg >> \bz\\ @VVV \Vert@. \\ \Pic(\X_F) @>\deg >> \bz \end{CD}\]
with surjective vertical map, the cokernel of both degree maps is $\bz/\delta\bz$ where $\delta$ is the g.c.d. of the degrees of all divisors on $\X_F$, i.e. the index of $\X_F$. Hence an exact sequence
\begin{equation} 0\to\Br(\overline{\X})\to H^1(\overline{S},P)\to\bz/\delta\bz\to 0.\label{h1comp}\end{equation}
We can similarly extend (\ref{trunc2}) to the Artin-Verdier compactification. The composite map
\[ (2\pi i\bz)\to Rf_{\bc,*}(2\pi i \bz)\to R^2f_{\bc,*}(2\pi i \bz)[-2] \xrightarrow{\deg} \bz[-2]\]
in $Sh(G_\br,S(\bc))$ clearly vanishes and the commutativity of the right hand square in (\ref{avdef}) then shows that the degree map on the middle row of (\ref{sbarcomp}) factors through the lower row, i.e. we obtain a commutative diagram with exact rows and columns
\begin{equation}\begin{CD}
P^{0,\overline{S}}@>>> R\phi^S_*P^0@>>> u^S_{\infty,*}K^0\\
@VVV @VVV @VVV\\
P^{\overline{S}}@>>> R\phi^S_*P@>>> u^S_{\infty,*}K\\
@VVV @VVV @VVV\\
\bz^{\overline{S}}@>>> R\phi^S_*\bz @>>> u^S_{\infty,*}\tau^{>0}R\hat{\pi}^S_*\bz\\
\end{CD}\label{sbarcomp2}\end{equation}
where the lower row is the defining exact triangle \cite{Flach-Morin-16}[App. A, Cor. 6.8] of the complex $\bz^{\overline{S}}$, and the left hand column defines $P^{0,\overline{S}}$. Note that in fact $\bz^{\overline{S}}=\bz$. The long exact cohomology sequence of the left hand column is
\begin{equation} 0\to H^0(\overline{S},P^0) \to H^0(\overline{S},P) \xrightarrow{\deg} \bz \to H^1(\overline{S},P^0) \to H^1(\overline{S},P)\to 0   \notag\end{equation}
since $H^1(\overline{S},\bz)=H^1(S,\bz)=0$. From (\ref{h0comp}) we obtain
\begin{equation} H^0(\overline{S},P^0)\cong \Pic^0(\X)/\Pic(\co_F)\label{h00comp}\end{equation}
and from (\ref{h1comp}) that $H^1(\overline{S},P^0)$ has cardinality $\#\Br(\overline{\X})\delta^2$.

In order to compute $H^i(\overline{S},P^0)$ in degrees $\geq 2$ we prove a torsion duality for $P^0/p^\nu$ for any prime $p$, the isomorphism $\mathbf{AV}(\overline{S},P^0)$ in diagram (\ref{avp0}) below. Following the proof of \cite{Flach-Morin-16}[Prop. 3.4] this then implies $$H^i(\overline{S},P^0)\cong H^{2-i}(\overline{S},P^0)^*$$ for $i\geq 2$, in particular $H^i(\overline{S},P^0)=0$ for $i\geq 3$. Following \cite{Flach-Morin-16}[Def. 3.6] we then define $R\Gamma_W(\overline{S},P^0)$ by an exact triangle
\[ R\Hom(R\Gamma_W(\overline{S},P^0),\bq[-2])\to R\Gamma(\overline{S},P^0)\to R\Gamma_W(\overline{S},P^0)\to \]
and obtain the cohomology groups described in Lemma \ref{h1br}. Conjecture ${\bf L}(\mathcal{X}_{et},1)$ entering into the definition of Weil-\'etale modifications in \cite{Flach-Morin-16}[Sec. 3] reduces to finiteness of $\Br(\X)=H^3(\X_{et},\bz(1))$ and finite generation of $\Pic(\X)$ which is well known.

Fix a prime number $p$ and consider the following commutative diagram with exact columns.
\begin{equation} \xymatrix@C=.07in{
& R\Gamma(\overline{S},\bz(1)/p^\nu)\ar[dl]\ar[rr]^{\mathbf{AV}(\overline{S},1)}\ar'[d][dd] & &R\Gamma(\overline{S},\bz/p^\nu)^*[-3]\ar[dl]\ar[dd]\\
R\Gamma(S,\bz(1)/p^\nu)\ar[dd]\ar[rr]^{\mathbf{AV}(S,1) } & &\hat{R}\Gamma_c(S,\bz/p^\nu)^*[-3]\ar[dd] &{}\\
& R\Gamma(\overline{\X},\bz(1)/p^\nu)\ar[dl]\ar'[d][dd]\ar'[r]^(.8){\mathbf{AV}(\overline{\mathcal{X}},1)}[rr] & &R\Gamma(\overline{\X},\bz(1)/p^\nu)^*[-5]\ar[dl]\ar[dd]\\
R\Gamma(\X,\bz(1)/p^\nu)\ar[dd]\ar[rr]^{\mathbf{AV}(\mathcal{X},1)} & &\hat{R}\Gamma_c(\X,\bz(1)/p^\nu)^*[-5]\ar[dd] &{}\\
& R\Gamma(\overline{S},P/p^\nu)[-2]\ar[dl]\ar'[r]^(.8){\mathbf{AV}(\overline{S},P)}[rr] & &R\Gamma(\overline{S},Q/p^\nu)^*[-5]\ar[dl]\\
R\Gamma(S,P/p^\nu)[-2]\ar[rr]^{\mathbf{AV}(S,P) } & &\hat{R}\Gamma_c(S,Q/p^\nu)^*[-5] &{}
}
\label{keydia}\end{equation}
The top left commutative square is induced by the top left square in (\ref{sbarcomp}) after applying $-\otimes^L\bz/p^\nu\bz$. The top and middle commutative square are those in the proof of \cite{Flach-Morin-16}[Thm. 6.24] for $S$ and $\X$, respectively. The front commutative diagram of exact triangles arises from Corollary \ref{compatibleduality} where we define the complex $Q^{\overline{S}}$ on $\overline{S}$ by the exact triangle
\begin{equation}  Q^{\overline{S}}\to R\overline{f}_*\bz(1)^{\overline{\X}}\xrightarrow{\deg}\bz^{\overline{S}}[-2]\to    \label{trunc3} \end{equation}
and denote by $Q$ its restriction to $S$. Note that there is also an exact triangle
\begin{equation}  \bz(1)^{\overline{S}}\to Q^{\overline{S}} \to P^{0,\overline{S}}[-2]\to.    \label{trunc4} \end{equation}
It follows then from Corollary \ref{compatibleduality} and \cite{Flach-Morin-16}[Thm. 6.24] that all arrows labeled $\mathbf{AV}$ in (\ref{keydia}) are quasi-isomorphisms. We finally obtain an isomorphism of exact triangles
\begin{equation}\begin{CD}
R\Gamma(\overline{S},\bz(1)/p^\nu) @>>> R\Gamma(\overline{S},Q/p^\nu) @>>> R\Gamma(\overline{S},P^0/p^\nu)[-2]\\
@V \mathbf{AV}(\overline{S},1) VV  @V {\mathbf{AV}(\overline{S},P)^*[-5]}VV @V{\mathbf{AV}(\overline{S},P^0)[-2]} VV\\
R\Gamma(\overline{S},\bz/p^\nu)^*[-3] @>>> R\Gamma(\overline{S},P/p^\nu)^*[-3] @>>> R\Gamma(\overline{S},P^0/p^\nu)^*[-3]
\end{CD}\label{avp0}\end{equation}
where the top, resp. bottom, row is induced by (\ref{trunc4}), resp. the left hand column in (\ref{sbarcomp2}), and the commutativity of the left square again follows from  Corollary \ref{compatibleduality}. This concludes the proof of Lemma \ref{h1br}.
\end{proof}

The $H^1$-part of the complex $R\Gamma_W(\mathcal{X}_{\infty},\mathbb{Z}(1))$ defined in \cite{Flach-Morin-16}[Def. 3.23] is specified by two exact triangles
\begin{equation} R\Gamma_W(S_{\infty},\mathbb{Z}(1))\to R\Gamma_W(\mathcal{X}_{\infty},\mathbb{Z}(1))\to R\Gamma_W(S_\infty, P)[-2]\label{tinfty1}\end{equation}
and
\begin{equation}R\Gamma_W(S_\infty, P^0) \to R\Gamma_W(S_\infty, P)\to R\Gamma_W(S_{\infty},\mathbb{Z})\label{tinfty2}\end{equation}
and its cohomology is computed by the following Lemma.

\begin{lemma} One has
\begin{equation}
H^i_W(S_\infty, P^0)=\begin{cases} H^1(\mathcal{X}(\mathbb{C}),(2\pi i)\mathbb{Z})^{G_\br} & i=-1\\
 \bigoplus\limits_{\text{$v$ real, $\X(F_v)\neq\emptyset$}}\Phi_v & i=0\\
\bigoplus\limits_{\text{$v$ real, $\X(F_v)=\emptyset$}}\bz/2\bz & i=1\\
0 & i\neq -1,0,1\end{cases}
\label{pinftycomp}\end{equation}
where $\Phi_v=J_F(F_v)/J_F(F_v)^0$ is the component group as defined in the introduction. Moreover
\begin{equation}\frac{\#H^0_W(S_\infty, P^0)}{\#H^1_W(S_\infty, P^0)}=\prod\limits_{\text{$v$ real}}\frac{\#\Phi_v}{\delta_v'\delta_v}.\label{periodindex}\end{equation}
\label{pinftylemma}\end{lemma}

\begin{proof} The topological space $\mathcal{X}_{\infty}=\X(\bc)/G_\br$ is a $2$-manifold (with boundary $\X(\br)$) and hence has sheaf-cohomological dimension $\leq 2$. The complex $$i_\infty^*\bz(1):=\tau^{\leq 1}R\pi_*(2\pi i)\bz$$ of \cite{Flach-Morin-16}[Def. 3.23] is concentrated in degrees $0$ and $1$ and $R^1\pi_*(2\pi i)\bz$ is supported on the closed subset $\X(\br)$, a union of circles, hence of cohomological dimension $1$. It follows that $R\Gamma_W(\mathcal{X}_{\infty},\mathbb{Z}(1))$ is concentrated in degrees $\leq 2$.

The two triangles (\ref{tinfty1}) and (\ref{tinfty2}) are direct sums over the infinite places $v\in S_\infty$ and we denote the respective direct summands by an index $v$. If $v$ is a complex place or a real place with $\X(F_v)\neq\emptyset$ then $f_{\infty,v}$ has a section and the
exact triangle \cite{Flach-Morin-16}[(47)]
\begin{equation}\label{inftydefining}
R\Gamma_W(\mathcal{X}_{\infty},\mathbb{Z}(1))_v\to R\Gamma(G_{\mathbb{R}},\mathcal{X}(\mathbb{C}),(2\pi i)\mathbb{Z})_v
\to R\Gamma(\mathcal{X}(\mathbb{R}),\tau^{>1}R\widehat{\pi}_*(2\pi i)\mathbb{Z})_v
\end{equation}
splits into a direct sum of its $H^i$-parts for $i=0,1,2$. The last term being concentrated in degrees $\geq 3$ we find
\begin{equation}\notag
R\Gamma_W(S_\infty, P^0)_v[-2]\cong \tau^{\leq 2}R\Gamma(G_\br,H^1(\mathcal{X}(F_v\otimes_\br\mathbb{C}),(2\pi i)\mathbb{Z})[-1])
\end{equation}
for the $H^1$-part. The group $H^1(G_\br,H^1(\mathcal{X}(F_v\otimes_\br\mathbb{C}),(2\pi i)\mathbb{Z}))$ is isomorphic to $\Phi_v$
as can be verified easily by taking $G_\br$-cohomology of the exponential sequence for the abelian variety $J_F(F_v\otimes_\br\mathbb{C})$.

If $v$ is a real place with $\X(F_v)=\emptyset$ an analysis of the spectral sequence
\[ H^p(G_{\mathbb{R}},H^q(\mathcal{X}(F_v\otimes_\br\mathbb{C}),(2\pi i)\mathbb{Z}))\Rightarrow H^{p+q}(G_{\mathbb{R}},\mathcal{X}(F_v\otimes_\br\mathbb{C}),(2\pi i)\mathbb{Z})\]
reveals that
\[ H^i_W(S_\infty, P)_v=\begin{cases}H^1(\mathcal{X}(F_v\otimes_\br\mathbb{C}),(2\pi i)\mathbb{Z})^{G_\br} & i=-1\\
 \ker\left(H^2(\mathcal{X}(F_v\otimes_\br\mathbb{C}),(2\pi i)\mathbb{Z})\cong\bz\twoheadrightarrow\bz/2\bz\right) & i=0. \end{cases}\]
Here one uses the fact that the end term is concentrated in degrees $\leq 2$ and that for $i\geq 1$
\[\#\Phi_v=\#H^i(G_\br,H^1(\mathcal{X}(F_v\otimes_\br\mathbb{C}),(2\pi i)\mathbb{Z}))\in\{1,2\}\]
by \cite{grossharris}[Prop. 3.3, Prop. 1.3]. The long exact sequence associated to (\ref{tinfty2}) then gives the computation (\ref{pinftycomp}).

To show (\ref{periodindex}) first note that $\X(F_v)=\emptyset$ if and only if $\delta_v=2$, otherwise $\delta_v=1$. Since $\delta_v'\mid\delta_v$ this shows (\ref{periodindex}) at places where $\X(F_v)\neq\emptyset$. If $\X(F_v)=\emptyset$ one has $\#\Phi_v=1$ or $2$ according to whether the genus $g$ of $\X$ is even or odd by \cite{grossharris}[Prop. 3.3] and one also has $\delta_v'=1$ or $2$ according to whether $g$ is even or odd by \cite{Lichtenbaum69}, \cite{PoonenStoll}[p. 1126]. This shows (\ref{periodindex}) at places where $\X(F_v)=\emptyset$.
\end{proof}

The complex $R\Gamma_{W,c}(S,P^0)$ is defined by an exact triangle
\[ R\Gamma_{W,c}(S,P^0)\to R\Gamma_W(\overline{S},P^0)\to R\Gamma_W(S_\infty,P^0)\to\]
and its cohomology, or at least the determinant of its cohomology, can easily be computed by combining Lemmas \ref{h1br} and  \ref{pinftylemma}. The $H^1$-part of the exact triangle (\ref{tangenttri}) is an exact triangle
\begin{equation}
H^1(\X,\co_\X)_\br[-3] \to \Pic^0(\X)_\br[-2]\oplus\Pic^0(\X)_\br[-3]\to R\Gamma_{W,c}(S,P^0)_\br[-2]\to.
\notag\end{equation}
The fundamental line of the $H^1$-part
\[ \Delta(S/\bz,P^0):={\det}_\bz R\Gamma_{W,c}(S,P^0)\otimes_\bz {\det}_\bz^{-1} H^1(\X,\co_\X)\]
is then trivialized using the period isomorphism
\begin{equation} H^1(\mathcal{X}(\mathbb{C}),(2\pi i)\mathbb{Z})^{G_\br}_\br\cong H^1(\X,\co_\X)_\br\label{period}\end{equation}
on $\X(\bc)$ and the intersection pairing
\begin{equation} \Pic^0(\X)_\br\cong \Hom_\bz(\Pic^0(\X),\br).\label{arakelovpairing}\end{equation}
One easily derives formula (\ref{BSDforX}) in the introduction. Note here that the determinant $\Omega(\X)$ of (\ref{period}) with respect to the natural $\bz$-structures on both sides includes possible torsion in $H^1(\X,\co_\X)$. This torsion subgroup is nontrivial if and only if $f$ is not cohomologically flat in dimension $0$.

Although it seems eminently plausible, we were not able to show that the pairing induced on
$$\Pic^0(\X)_\br=\im(H^2_c(\X,\br(1))\to H^2(\X,\br(1)))$$
by the pairing \cite{Flach-Morin-16}[Prop. 2.1]
$$ H^2(\X,\br(1))\otimes H^2_c(\X,\br(1))\to H^4_c(\X,\br(2))\to\br$$
coincides with the Arakelov intersection pairing, discussed for example in \cite{Hriljac}. We shall however assume this from now on, or equivalently, we simple restate our conjecture using the Arakelov intersection pairing. Since the Arakelov intersection pairing is negative definite on $\Pic^0(\X)_\br$ \cite{Hriljac}[Thm.3.4, Prop.3.3], hence nondegenerate, the map (\ref{arakelovpairing}) is indeed an isomorphism.

\section{Comparison to the Birch and Swinnerton-Dyer Conjecture}\label{sec:comparison}

The comparison of our conjecture (\ref{BSDforX}) with the Birch and Swinnerton-Dyer formula involves two key ingredients. The first is a precise formula relating the cardinalities of $\Br(\X)$ and $\Sha(J_F)$ and the second a Lemma about the behavior of the Arakelov intersection pairing in exact sequences, Lemma \ref{innerprod} below. The essential ideas for the comparison of $\Br(\X)$ and $\Sha(J_F)$ can already be found in Grothendieck's article \cite{grothbrauer} and his results imply that $\Br(\X)\cong\Sha(J_F)$ if, for example, $f$ has a section and $F$ is totally imaginary. The general case is considerably more complicated and has been studied by a number of authors until it was recently settled by Geisser \cite{Geisser17}. Unfortunately, Geisser's formula still has the condition that $F$ is totally imaginary, so we give here a generalization of his result without this condition and with a different proof. What makes both proofs eventually possible are the duality results of Saito in \cite{Saito89}.

For any place $v$ of $F$ we denote by $\delta_v$, resp. $\delta'_v$, the index, resp. period of $\X_{F_v}$ over $F_v$, i.e. the cardinalities of the cokernel of $\Pic(\X_{F_v})\xrightarrow{\deg}\bz$, resp. $P(F_v)\xrightarrow{\deg}\bz$.
%If $v$ is nonarchimedean we denote by $d_v$ the index of $\X_{F^{ur}_v}$ over $F^{ur}_v$, where $F^{ur}_v/F_v$ is the maximal unramified extension, and if $v$ is archimedean we set $d_v=\delta_v'$.
Then one has
\begin{equation} \delta_v'\mid\delta_v\mid\delta
\notag\end{equation}
and $\delta_v/\delta_v'\in\{1,2\}$ for all places $v$ \cite{PoonenStoll}[p.1126]. We define
\begin{align} \alpha:=&\#\cok\left(\Pic^0(\X_F)\to J_F(F)\right)\notag\\
  =&\#\cok\left(H^0(\overline{S},P^0)\to J_F(F)\right) \label{alphadef}\end{align}
where the equality holds since $H^0(\overline{S},P^0)=\Pic^0(\X)/\Pic(\co_F)\to\Pic^0(\X_F)$ is surjective.
\begin{prop} If $\Br(\X)$ is finite then
\begin{equation} \#\Br(\overline{\X})\delta^2 = \frac{\prod_v\delta_v'\delta_v}{\alpha^2} \#\Sha(J_F)  \label{Geisserformula}\end{equation}
where the product is over all places $v$ of $F$ and $\Br(\overline{\X})$ was defined in Lemma \ref{h1br}.
\label{Geisserprop}\end{prop}

\begin{remark} As in \cite{Geisser17}[Cor. 1.2] it follows that the cardinality of $\Br(\overline{\X})$ is a square if it is finite.
\end{remark}

\begin{proof} For each finite place $v$ of $F$ we denote by $S_v=\overline{S}_v$ the spectrum of the Henselization of $S$ at the closed point $v$ and by $F_v$ its field of fractions. If $v\in S_\infty$ is an infinite place of $F$ we denote by $F_v\subseteq \overline{F}$ the fixed field of a chosen decomposition group of $v$ in $\Gal(\overline{F}/F)$ and by $\overline{S}_v$ the local topos that is glued from $\Spec(F_v)_{et}$ and $\mathrm{Shv}(v)=\Set$. Then there is a morphism of topoi $\overline{S}_v\to\overline{S}_{et}$ and for a complex of sheaves $\F$ on $\overline{S}_{et}$ the group $H^i(\overline{S}_v,\F)=H^i(v,\F)=\mathcal H^i(\F)_v$ is the stalk of $\F$ at the point $v$ of the topos $\overline{S}_{et}$.

Let $\eta:\Spec(F)\to\overline{S}$ be the inclusion of the generic point and define a commutative diagram of complexes on $\overline{S}$ with exact rows and columns
\begin{equation}\begin{CD} E@>>> P^{0,\overline{S}} @>>> \tilde{P}^0\\
\Vert@. @VVV @VVV \\
E@>>>  P^{\overline{S}} @>>> \eta_*\eta^*P^{\overline{S}}\\
@. @V\deg VV @V\deg VV\\
{} @. \bz @=\bz
\end{CD}\label{genfibre}\end{equation}
Note that $\eta^*P^{0,\overline{S}}$ is the sheaf represented by the Jacobian $J_F$ and $\mathcal H^0(\tilde{P}^0)=\eta_*J_F$ the sheaf represented by the Neron model of $J_F$ over $S$.

\begin{lemma} The complex $E$ is a sum of skyscraper sheaves in degrees $0$ and $1$. One has
\[ \mathcal H^0(E)\cong \bigoplus_{v\in \Sigma_f}E_v\]
where $\Sigma_f$ is the set of finite places of $F$ where $f$ is not smooth, and
 \[\mathcal H^1(E)\cong \bigoplus_{v\in S_\infty} \bz/\delta_v\bz.\]
Moreover, for $v\in\Sigma_f$ one has $H^1(S_v,E_v)\cong \bz/\delta_v\bz$ and hence
\begin{equation} H^1(\overline{S},E)\cong \bigoplus\limits_{v\in\Sigma}\bz/\delta_v\bz \label{h1ecompute}\end{equation}
where $\Sigma=\Sigma_f\cup S_\infty$.
\label{elemma}\end{lemma}

\begin{proof} The restriction of the middle row of (\ref{genfibre}) to $S$ is a short exact sequence of sheaves concentrated in degree $0$ and has been analyzed by Grothendieck \cite{grothbrauer}[(4.10 bis)]. The restriction of $E$ to $S$ is a sum over $v\in\Sigma_f$ of skyscraper sheaves $E_v$ placed in degree $0$. Viewing $E_v$ as a $G_{\kappa(v)}$-module there is an exact sequence
\begin{equation} 0\to \bz \to \sum_{i\in C_v}\Ind_{G_{\kappa(v)_i}}^{G_{\kappa(v)}}\bz\to E_v \to 0\label{evdef}\end{equation}
where $C_v$ is the set of irreducible components of the fibre $\X_{\kappa(v)}$ and $\kappa(v)_i$ denotes the algebraic closure of $\kappa(v)$ in the function field of the component $\X_{\kappa(v),i}$ indexed by $i\in C_v$. By \cite{grothbrauer}[(4.25)]
\[ H^1(\overline{S},E_v)\cong H^1(S_v,E_v)\cong \bz/\gcd(r_{v,i}d_{v,i})\bz\]
where $r_{v,i}=[\kappa(v)_i:\kappa(v)]$ and $d_{v,i}$ is the multiplicity of $\X_{\kappa(v),i}$ in the fibre. By \cite{Bosch-Liu}[Rem. 1.4] and \cite{Colliot-Saito}[Thm. 3.1] one has
\begin{equation}\gcd(r_{v,i}d_{v,i})=\delta_v.\label{liu-colliot-saito}\end{equation}
The middle row in (\ref{genfibre}), the fact that $\eta_*\eta^*P^{\overline{S}}$ is concentrated in degree $0$ and Lemma \ref{pbar} below imply $\mathcal H^i(E)=0$ for $i\geq 2$ and an exact sequence
\[ 0\to \mathcal H^0(P^{\overline{S}})\to\mathcal \eta_*\eta^*P^{\overline{S}}\to\mathcal H^1(E)\to 0.\]
We saw already that the restriction of $\mathcal H^1(E)$ to $S$ is zero. If
$$\Spec(F)\xrightarrow{\tilde{\eta}}U\xrightarrow{j} S\xrightarrow{\phi^S}\overline{S}$$
is an open subscheme over which $f$ is smooth we have $\tilde{\eta}_*\tilde{\eta}^*P=P$ and hence
\[u^{S,*}_\infty\eta_*\eta^*P^{\overline{S}}=u^{S,*}_\infty \phi^S_*j_*P=u^{S,*}_\infty\phi^S_*P.\]
It then follows from the bottom triangle in (\ref{sbarcomp}) that
$$u^{S,*}_\infty \mathcal H^1(E)\cong \mathcal H^2(K)\cong\ker\left( \bigoplus\limits_{v\in S_\infty}\Br(F_v) \to \bigoplus\limits_{v\in S_\infty}\Br(\X_{F_v})\right)\cong \bigoplus\limits_{v\in S_\infty}\bz/\delta_v\bz. $$
\end{proof}

\begin{lemma} For $i\geq 1$ one has
\begin{equation} \mathcal H^i(P^{\overline{S}})=0.\notag\end{equation}
\label{pbar}\end{lemma}

\begin{proof} Using the long exact sequence induced by the left column in (\ref{sbarcomp}) and the fact that
$\bz(1)^{\overline{S}}=\phi_*^S\bg_m[-1]$ is concentrated in degree $1$ \cite{Flach-Morin-16}[Prop. 6.11] it suffices to show
\begin{equation} \mathcal H^i(R\overline{f}_*\bz(1)^{\overline{\X}})=0 ; \quad i\geq 3.\notag\end{equation}
Using the long exact sequence induced by the central row in (\ref{sbarcomp}) it suffices to show that
\begin{equation}R\phi^S_*Rf_*\bz(1) \to u^S_{\infty,*}Rf_{\infty,*}\tau^{>1}R\hat{\pi}_*(2\pi i\bz)\notag\end{equation}
is a surjection in degree $2$ (this is clear as the target is zero) and an isomorphism in degrees $\geq 3$. Since $Rf_*\bz(1)$ is concentrated in degrees $\leq 2$ this is equivalent to
\begin{equation}\tau^{\geq 3}u^{S,*}_{\infty}R\phi^S_*Rf_*\bz(1) \to Rf_{\infty,*}\tau^{>1}R\hat{\pi}_*(2\pi i\bz)\label{rfres}\end{equation}
being an isomorphism. But the map in (\ref{rfres}) is isomorphic to the map
\[R\Gamma(G_{\mathbb{R}},\mathcal{X}(\mathbb{C}),(2\pi i)\mathbb{Z})
\to R\Gamma(\mathcal{X}(\mathbb{R}),\tau^{>1}R\widehat{\pi}_*(2\pi i)\mathbb{Z})\]
in (\ref{inftydefining}) which we have shown to be an isomorphism in degrees $\geq 3$ in the proof of Lemma \ref{pinftylemma}.
To see that the two maps are isomorphic in degrees $\geq 3$ use the exact triangles
\[ Rf_{\bc,*}(2\pi i)\bz \to  Rf_{\bc,*}(2\pi i)\bq \to  Rf_{\bc,*}(2\pi i)\bq/\bz\]
and
\[ Rf_*\bz(1)\to  Rf_*\bq(1)\to  Rf_*\bq/\bz(1)\]
and the isomorphism $\alpha^{S,*}Rf_*\bq/\bz(1)\cong Rf_{\bc,*}(2\pi i)\bq/\bz$ arising from proper base change and (a $G_\br$-equivariant version of) Artin's comparison theorem between \'etale and analytic cohomology, where $\alpha^S$ is the composite morphism of topoi
\[ Sh(G_\br,S(\bc))\to S_{\br,et}\to S.\]
By \cite{Flach-Morin-16}[Lem. 6.2] we have $u^{S,*}_\infty R\phi^S_*=R\pi^S_*\alpha^{S,*}$ and hence we find
\[u^{S,*}_{\infty}R\phi^S_*Rf_*\bq/\bz(1)\cong R\Gamma(G_{\mathbb{R}},\mathcal{X}(\mathbb{C}),(2\pi i)\bq/\mathbb{Z}).\]
The long exact sequences induced by the above two triangles then show that
\[\mathcal H^i\left(u^{S,*}_{\infty}R\phi^S_*Rf_*\bz(1)\right)\cong H^i(G_{\mathbb{R}},\mathcal{X}(\mathbb{C}),(2\pi i)\mathbb{Z})\]
for $i\geq 3$ where for $i=3$ we also need the fact that
\[\mathcal H^2\left(u^{S,*}_{\infty}R\phi^S_*Rf_*\bq(1)\right)\to H^2(G_{\mathbb{R}},\mathcal{X}(\mathbb{C}),(2\pi i)\mathbb{Q})\cong \bq\]
is surjective. This is clear since this map is just the degree map.
\end{proof}

We continue with the proof of Prop. \ref{Geisserprop}. Taking cohomology over $\overline{S}$ and $\overline{S}_v$ of the top row in (\ref{genfibre}) gives a map of long exact sequences
\[ \minCDarrowwidth1em\begin{CD}
H^0(\overline{S},P^0) @>\phi_{-1}>> J_F(F) @>>> H^1(\overline{S},E) @>\phi_0>> H^1(\overline{S},P^0)@>\phi_1>> \\
@VVV @VVV \Vert@. @VVV\\
\bigoplus\limits_{v\in\Sigma}H^0(\overline{S}_v,P^0) @>>> \bigoplus\limits_{v\in\Sigma}J_F(F_v) @>>> H^1(\overline{S},E) @>\phi_5>> \bigoplus\limits_{v\in\Sigma}H^1(\overline{S}_v,P^0)@>>>
\end{CD}\]
where the vertical isomorphism holds since $E$ is a sum of skyscraper sheaves supported in $\Sigma$. From Lemma \ref{h1br}, (\ref{h1ecompute}) and (\ref{alphadef}) we have
\begin{equation}
\#\Br(\overline{\X})\delta^2=\#H^1(\overline{S},P^0)=\frac{\#H^1(\overline{S},E)}{\#\cok(\phi_{-1})}\cdot\#\im(\phi_1)=\frac{\prod_{v\in\Sigma}\delta_v}{\alpha}\cdot\#\im(\phi_1).
\label{compbeginning}\end{equation}
To compute $\#\im(\phi_1)$  consider the continuation of the long exact sequences
\[ \minCDarrowwidth1em\begin{CD}
 H^1(\overline{S},E) @>\phi_0>> H^1(\overline{S},P^0)@>\phi_1>> H^1(\overline{S},\tilde{P}^0)@>>> H^2(\overline{S},E)\\
\Vert@. @VV\phi_3 V @VV\phi_2 V \Vert@. \\
H^1(\overline{S},E) @>\phi_5>> \bigoplus\limits_{v\in\Sigma}H^1(\overline{S}_v,P^0)@>\phi_4>> \bigoplus\limits_{v\in\Sigma}H^1(\overline{S}_v,\tilde{P}^0)@>>> H^2(\overline{S},E)\\
@. @VVV @. @.\\
{}@.  H^2(\overline{S},j_!j^*P^0) @.{}@.\\
@. @VV\phi_6 V @. @.\\
{}@.  H^2(\overline{S},P^0) @.{}@.
\end{CD}\]
where the vertical column is also exact. Here $j$ is the open immersion $\overline{S}\setminus\Sigma \xrightarrow{j} \overline{S}$.
Since $\ker(\phi_2)\subseteq\im(\phi_1)$ we have
\begin{align*} \#\im(\phi_1)=\#\im(\phi_2\circ\phi_1)\cdot\#\ker(\phi_2)=\#\im(\phi_4\circ\phi_3)\cdot\#\ker(\phi_2)
\end{align*}
and since $\im(\phi_5)\subseteq\im(\phi_3)$ we obtain
\begin{align} \#\im(\phi_1)=&\#\im(\phi_4\circ\phi_3)\cdot\#\ker(\phi_2)=\frac{\#\im(\phi_3)}{\#\im(\phi_5)}\cdot\#\ker(\phi_2)\notag\\
=&\frac{\prod_{v\in\Sigma}\#H^1(\overline{S}_v,P^0)}{\#\ker(\phi_6)\#\im(\phi_5)}\cdot\#\ker(\phi_2)\notag\\
=&\frac{\prod_{v\in\Sigma}\delta_v}{\alpha\prod_{v\in \Sigma}(\delta_v/\delta_v')}\cdot\#\Sha(J_F)\label{compending}
\end{align}
using Lemmas \ref{l2}, \ref{l3}, \ref{l4} below. Combining (\ref{compbeginning}) and (\ref{compending}) finishes the proof.

\end{proof}

\begin{lemma} For $v\in\Sigma$ one has
\begin{equation}H^1(\overline{S}_v,P^0)\cong\bz/\delta_v\bz\label{h1p0}\end{equation}
and $\im(\phi_{5,v})\cong\bz/(\delta_v/\delta_v')\bz$ where $\phi_5=\oplus_v\phi_{5,v}$.
\label{l2}\end{lemma}

\begin{proof} Consider the commutative diagram with exact rows and columns induced by (\ref{genfibre})
\[\minCDarrowwidth1em\begin{CD} H^0(\overline{S}_v,P^0) @>\psi>> J_F(F_v) @>>> H^1(\overline{S}_v,E)@>\phi_{5,v}>>{}\\
 @VVV @VVV \Vert@. \\
H^0(\overline{S}_v,P) @>>> P(F_v) @>>> \bz/\delta_v\bz @>>> H^1(\overline{S}_v,P)=0\\
 @V\deg VV @V\deg VV\\
 \bz @=\bz\\
 @VVV @VVV\\
 H^1(\overline{S}_v,P^0)@>>> \bz/\delta_v'\bz\\
 @VVV @VVV\\
 H^1(\overline{S}_v,P)=0 @>>> 0
\end{CD}\]
The vanishing of $H^1(\overline{S}_v,P)$ is proven in \cite{grothbrauer}[(4.15)] for $v\in\Sigma_f$ and follows from Lemma \ref{pbar} for $v\in S_\infty$. The image of $H^0(\overline{S}_v,P)$ in $P(F_v)$ coincides with $\Pic(\X_{F_v})$. For $v\in\Sigma_f$ this is because $H^0(\overline{S}_v,P)=\Pic(\X_{S_v})\to \Pic(\X_{F_v})$ is surjective and for $v\in S_\infty$ both subgroups coincide with the kernel of the map $P(F_v)\to \Br(F_v)$. Hence the degree map on $H^0(\overline{S}_v,P)$ has image $\delta_v\bz$ which gives (\ref{h1p0}). By the snake Lemma on finds $\cok(\psi)\cong\bz/\delta'_v\bz$ and $\im(\phi_{5,v})\cong\bz/(\delta_v/\delta_v')\bz$.
\end{proof}

\begin{lemma} One has $\#\ker(\phi_6)=\alpha$.  \label{l3}\end{lemma}
\begin{proof} For the open subscheme $U:=S\setminus\Sigma_f$ of $S$ and any prime $p$ the proof of Prop. \ref{avprop} generalizes to prove a duality isomorphism $\mathbf{AV}(\mathcal{X}_U,1)$ fitting into a commutative diagram of isomorphisms
\[\begin{CD}
R\Gamma(U,\bz(1)/p^\nu) @>>> R\Gamma(\X_U,\bz(1)/p^\nu) \\
@V \mathbf{AV}(U,1) VV  @V\mathbf{AV}(\mathcal{X}_U,1) VV \\
\hat{R}\Gamma_c(U,\bz/p^\nu)^*[-3] @>>> \hat{R}\Gamma_c(\X_U,\bz(1)/p^\nu)^*[-5].
\end{CD}\]
One then obtains diagrams (\ref{keydia}) with $S$, resp. $\X$, replaced by $U$, resp. $\X_U$. The triangle (\ref{trunc4}) and the left hand column in (\ref{sbarcomp2}) then induce an isomorphism $\mathbf{AV}(U,P^0)$ fitting into a commutative diagram
\begin{equation}\begin{CD}
R\Gamma(\overline{S},P^0/p^\nu) @>>> R\Gamma(U,P^0/p^\nu)\\
@V{\mathbf{AV}(\overline{S},P^0)} VV @V{\mathbf{AV}(U,P^0)}VV \\
R\Gamma(\overline{S},P^0/p^\nu)^*[-1]@>>>\hat{R}\Gamma_c(U,P^0/p^\nu)^*[-1].
\end{CD}\notag\end{equation}
One has isomorphisms
\[ H^2(\overline{S},j_!P^0)\cong\hat{H}^2_c(U,P^0)\cong \hat{H}^1_c(U,P^0\otimes\bq/\bz)\cong H^0(U,P^0)^*\]
where the first isomorphism holds since Tate cohomology agrees with ordinary cohomology in degrees $\geq 1$, the second since
$\hat{H}^i_c(U,P^0)$ is torsion for $i=1,2$ and the third by taking the limit of $\mathbf{AV}(U,P^0)^*$ over all $p$ and all $\nu$. One finds that $\phi_6$ is dual to the natural restriction map $$\phi_6^*:H^0(\overline{S},P^0)\cong\Pic^0(\X)/\Pic(\co_F)\to H^0(U,P^0)\cong J_F(F)$$ where the last
isomorphism holds since $f$ is smooth over $U$, hence $P^0$ coincides with the (sheaf represented by the) Neron model of $J_F$. From the definition (\ref{alphadef}) of $\alpha$ we conclude that $\alpha=\#\cok(\phi_6^*)=\#\ker(\phi_6)$.
\end{proof}

\begin{lemma} There is an isomorphism $\ker(\phi_2)\cong \Sha(J_F)$.     \label{l4}\end{lemma}
\begin{proof} Consider the commutative diagram with exact rows and columns
\[\begin{CD}
{}@.  0 @. 0\\
@. @VVV @VVV\\
\ker(\psi) @>>> H^1(\overline{S},\eta_*J_F) @>\psi>> \bigoplus\limits_v H^1(\overline{S}_v,\eta_*J_F)\\
@VVV @VVV @VVV\\
\Sha(J_F) @>>> H^1(F,J_F) @>>> \bigoplus\limits_v H^1(F_v,J_F)\\
@. @VVV @VVV\\
{}@.     H^0(\overline{S},R^1\eta_*J_F) @= \bigoplus\limits_v H^0(\overline{S}_v, R^1\eta_*J_F)
\end{CD}\]
where the vertical exact sequences arise from the Leray spectral sequence for the morphism $\eta$. Note that $R^1\eta_*J_F$ is a sum of skyscraper sheaves (with stalk $H^1(I_v,J_F)$) hence the bottom identity. An easy diagram chase then shows that $\ker(\psi)\cong\Sha(J_F)$. Exactly the same argument applies to the diagram
\[\begin{CD}
{}@.  0 @. 0\\
@. @VVV @VVV\\
\ker(\psi) @>>> H^1(\overline{S},\eta_*J_F) @>\psi>> \bigoplus\limits_v H^1(\overline{S}_v,\eta_*J_F)\\
@VVV @VVV @VVV\\
\ker(\phi_2) @>>> H^1(\overline{S},\tilde{P}^0) @>\phi_2>> \bigoplus\limits_v H^1(\overline{S}_v,\tilde{P}^0)\\
@. @VVV @VVV\\
{} @. H^0(\overline{S},\mathcal H^1(\tilde{P}^0)) @= \bigoplus\limits_v H^0(\overline{S}_v,\mathcal H^1(\tilde{P}^0))
\end{CD}\]
where now the vertical maps arise from the hypercohomology spectral sequence for the complex $\tilde{P}^0$, noting that $\mathcal H^0(\tilde{P}^0)\cong\eta_*J_F$ and that $\mathcal H^1(\tilde{P}^0)$ is a sum of skyscraper sheaves over points in $\Sigma$. This finishes the proof.
\end{proof}

By an inner product on a finitely generated abelian group $N$ we mean a (either positive or negative) definite, symmetric bilinear form $<-,->$ on $N_\br:=N\otimes_\bz\br$. For such an $N$ we define
\begin{equation}  \Delta(N):=\frac{|\det(<b_i,b_j>)|}{[N:\bigoplus_i\bz b_i]^2}\label{vdef}\end{equation}
where $\{b_i\}\subseteq N$ is a maximal linearly independent subset. One easily verifies that $\Delta(N)$ only depends on $<-,->$ but not on the choice of  $\{b_i\}$.

\begin{lemma} Let
\[ \cdots \to N_i \xrightarrow{d_i} N_{i+1} \to \cdots\]
be an exact sequence of finitely generated abelian groups of finite length. Assume each $N_i$ is equipped with an inner product $\tau_i$ so that $$\tau_{i+1}\vert_{\im(d_i)_\br}=\tau_i\vert_{\ker(d_i)_\br^\perp}$$
via the isomorphism $d_i:\ker(d_i)_\br^\perp\cong \im(d_i)_\br$ induced by $d_i$. Then
\begin{equation}\prod_i \Delta(N_i)^{(-1)^i}=1.\notag\end{equation}
\label{innerprod}\end{lemma}

\begin{proof} If suffices to prove the statement for the short exact sequences
\[ 0\to \ker(d_i)\to N_i\to \im(d_i)\to 0.\]
A maximal linearly independent subset $\{b_j\}_{1\leq j\leq l}\subset \ker(d_i)$ can be extended to a maximal linearly independent subset
$\{b_j\}_{1\leq j\leq l+k}\subset N_i$ and $\{\bar{b}_j\}_{l+1\leq j\leq l+k}\subset \im(d_i)$ is then also a maximal linearly independent subset. By the snake lemma we have
\[  [N_i:\bigoplus_{j=1}^{l+k}\bz b_j]=[\ker(d_i):\bigoplus_{j=1}^{l}\bz b_j]\cdot[\im(d_i):\bigoplus_{j=l+1}^{l+k}\bz \bar{b}_j]\]
and we also have
\[ \det(<b_j,b_{j'}>)_{1\leq j,j'\leq l+k}=\det(<b_j,b_{j'}>)_{1\leq j,j'\leq l}\cdot \det(<\bar{b}_j,\bar{b}_{j'}>)_{l+1\leq j,j'\leq l+k}\]
since the $\{b_j\}_{l+1\leq j\leq l+k}$ can be modified into elements of $\ker(d_i)_\br^\perp$ (by adding elements of $\ker(d_i)_\br$) without changing $\det(<\bar{b}_j,\bar{b}_{j'}>)_{l+1\leq j,j'\leq l+k}$ or $\det(<b_j,b_{j'}>)_{1\leq j,j'\leq l+k}$.
\end{proof}

Recall that for each $v\in\Sigma_f$ we denote by $C_v$ the set of irreducible components of the fibre $\X_{\kappa(v)}$ and by $r_{v,i}=[\kappa(v)_i:\kappa(v)]$ the degree of the constant field of the component corresponding to $i\in C_v$.

\begin{lemma} Let $L(J_F,s)$ be the Hasse-Weil L-function of the Jacobian of $\X_F$. Then
\begin{equation} \ord_{s=1}\zeta(H^1,s)=\ord_{s=1}L(J_F,s)+\sum_{v\in\Sigma_f}(\#C_v-1)\label{vancomp}\end{equation}
and
\begin{equation} \zeta^*(H^1,1)=L^*(J_F,1)\cdot\prod_{v\in\Sigma_f}(\log Nv)^{\#C_v-1}\prod_{i\in C_v}r_{v,i}\label{leadcomp}\end{equation}
\label{lfunction}\end{lemma}

\begin{proof} Denoting by $\Y_0$ the set of closed points of a scheme $\Y$ we have
\begin{align*} \zeta(\X,s)=&\prod_{x\in\X_0}\frac{1}{1-\#\kappa(x)^{-s}}=\prod_{v\in S_0}\zeta(\X_{\kappa(v)},s)\\
=&\prod_{v\in S_0}\prod_{m=0}^2\mydet_{\bq_l}\left(\id-Nv^{-s}\Frob_v\vert H^m(\X_{\overline{\kappa(v)}},\bq_l)\right)^{(-1)^{m+1}}
\end{align*}
and isomorphisms \cite{Bloch87}[Lem. 1.2]
\begin{align*} H^0(\X_{\overline{\kappa(v)}},\bq_l)& \cong H^0(\X_{\overline{F}},\bq_l)^{I_v}\cong\bq_l\\
H^1(\X_{\overline{\kappa(v)}},\bq_l)& \cong H^1(\X_{\overline{F}},\bq_l)^{I_v}
\end{align*}
and an exact sequence
\[ 0 \to E_v\otimes_\bz\bq_l(-1)\to H^2(\X_{\overline{\kappa(v)}},\bq_l) \to H^2(\X_{\overline{F}},\bq_l)^{I_v}\cong\bq_l(-1)\to 0. \]
From (\ref{evdef}) we have
\[\mydet_{\bq_l}\left(\id-T\cdot\Frob_v\vert E_v\otimes_\bz\bq_l(-1)\right)=\frac{1}{1-T\cdot Nv}\prod_{i\in C_v}(1-(T\cdot Nv)^{r_{v,i}}).\]
Comparing with (\ref{factorization}) we find (see \cite{Flach-Morin-16}[Rem. 7.5] for connections with the perverse $t$-structure)
\begin{align*} \zeta(H^1,s)=&\prod_{v\in S_0}\mydet_{\bq_l}\left(\id-Nv^{-s}\Frob_v\vert H^1(\X_{\overline{F}},\bq_l)^{I_v}\right)^{-1}\\
\cdot&\prod_{v\in\Sigma_f} \mydet_{\bq_l}\left(\id-Nv^{-s}\Frob_v\vert E_v\otimes_\bz\bq_l(-1)\right)\\
=&L(J_F,s)\cdot\prod_{v\in\Sigma_f}\frac{1}{1-Nv^{-(s-1)}}\prod_{i\in C_v}(1-Nv^{-(s-1)r_{v,i}}).
\end{align*}
From this (\ref{vancomp}) and (\ref{leadcomp}) are immediate.
\end{proof}

\begin{lemma} For each $v\in\Sigma_f$ we have
\[ \Delta(E_v^{G_{\kappa(v)}})=\frac{\#\Phi_v}{\delta_v'\delta_v}(\log Nv)^{\#C_v-1}\prod_{i\in C_v}r_{v,i}\]
where $\Phi_v$ is the component group of the Neron model of $J_F$ at $v$. Here $\Delta$ is formed with respect to the Arakelov intersection pairing.
\label{evintersection}\end{lemma}

\begin{proof} Recall that we denote by $d_{v,i}$ the multiplicity of the irreducible component $\X_{\kappa(v),i}$ in the fibre $\X_{\kappa(v)}$. Let $\bz^{C_v}$ be the group of divisors on $\X$ supported in $\X_{\kappa(v)}$ and denote by $<D_1,D_2>=\deg\co(D_1)\vert_{D_2}$ the intersection number of two divisors.

Taking $G_{\kappa(v)}$-cohomology of (\ref{evdef}) we obtain a short exact sequence
\[ 0\to \bz\xrightarrow{\gamma} \bz^{C_v} \to E_v^{G_{\kappa(v)}}\to H^1(G_{\kappa(v)},\bz)=0 \]
where $\gamma:1\mapsto \X_{\kappa(v)}=\sum_i d_{v,i}\,\X_{\kappa(v),i}$ by \cite{grothbrauer}[(4.22)]. So $\im(\gamma)$  lies in the image of $\Pic(\co_F)$ and $\tilde{E}_v:=E_v^{G_{\kappa(v)}}$ is the subgroup of $\Pic^0(\X)/\Pic(\co_F)$ generated by divisors supported in $\X_{\kappa(v)}$. Denoting by $N^\star=\Hom_\bz(N,\bz)$ the $\bz$-dual define a complex
\[ \bz^{C_v}\xrightarrow{\tilde{\alpha}}\bz^{C_v}\cong (\bz^{C_v})^\star\xrightarrow{\gamma^\star}\bz^\star\cong\bz\]
where
\[\tilde{\alpha}(D)=\sum_i <D,\X_{\kappa(v),i}>\X_{\kappa(v),i}.\]
Since the intersection pairing is trivial on the image of $\Pic(\co_F)$  the map $\tilde{\alpha}$ factors through $\tilde{E}_v$ and
\[\ker(\gamma^\star)/\im(\tilde{\alpha})\cong\tilde{E}_v^\star/\tilde{\alpha}(\tilde{E}_v)\] has cardinality
$|\det(<b_i,b_j>)|$ where $\{b_i\}$ is a $\bz$-basis of $\tilde{E}_v/(\tilde{E}_v)_{tor}$. Since
$$(\tilde{E}_v)_{tor}\cong\bz/d_v\bz;\quad\quad d_v:=\gcd(d_{v,i})$$
and since the Arakelov intersection pairing is given by $<-,->\log Nv$ \cite{Hriljac} we find
\begin{equation} \frac{\Delta(E_v^{G_{\kappa(v)}})}{(\log Nv)^{\#C_v-1}}=\frac{\Delta(\tilde{E}_v)}{(\log Nv)^{\#C_v-1}}=\frac{|\det(<b_i,b_j>)|}{d_v^2}=\frac{\#\ker(\gamma^\star)/\im(\tilde{\alpha})}{d_v^2}.\label{delta1}\end{equation}

On the other hand, by \cite{Bosch-Liu}[Thm. 1.11] and \cite{LLR}[proof of Lem. 4.4] there is an exact sequence
\begin{equation} 0\to\ker(\beta)/\im(\alpha) \to \Phi_v\to \bz/(\delta'_v/d_v)\bz\to 0\label{delta4}\end{equation}
where
\[ \bz^{C_v} \xrightarrow{\alpha} \bz^{C_v}\xrightarrow{\beta} \bz\]
are the maps
\[\alpha(D)=\sum_i r_{v,i}^{-1}<D,\X_{\kappa(v),i}>\X_{\kappa(v),i}\quad\quad\beta(\X_{\kappa(v),i})=r_{v,i}d_{v,i}.\]

We have $\tilde{\alpha}=r\alpha$ where $r:\bz^{C_v}\to\bz^{C_v}$ is the map $(n_i)\mapsto (r_{v,i}n_i)$. The Snake Lemma applied
to
\[\begin{CD} 0 @>>> \bz^{C_v} @>r>> \bz^{C_v} @>>> \prod_i\bz/r_{v,i}\bz @>>> 0\\
@. @V\beta VV @V\gamma^\star VV @VVV @.\\
0 @>>> \bz @>\sim >> \bz @>>> 0 @>>> 0
\end{CD}\]
gives an exact sequence
\[ 0\to\ker(\beta)\xrightarrow{r}\ker(\gamma^*) \to R \to 0\]
where
\begin{equation} \#R =\frac{(\prod_i r_{v,i})\#\cok(\gamma^*)}{\#\cok(\beta)}=\frac{(\prod_i r_{v,i}) d_v}{\gcd(r_{v,i}d_{v,i})}
=\frac{\prod_i r_{v,i}}{(\delta_v/d_v)}\label{delta3}\end{equation}
using (\ref{liu-colliot-saito}). Hence
\begin{equation} \#\ker(\gamma^\star)/\im(\tilde{\alpha})=\#\ker(\gamma^\star)/r(\im(\alpha))=\#R\cdot\#\ker(\beta)/\im(\alpha).  \label{delta2}\end{equation}
Combining (\ref{delta1}),(\ref{delta2}),(\ref{delta3}) and (\ref{delta4}) finishes the proof of Lemma \ref{evintersection}.
\end{proof}

\begin{lemma} Let $\mathcal J$ be the Neron model of $J_F$ and $\Omega(\mathcal J)$ the determinant of the period isomorphism between
the free, finitely generated abelian groups $H^1(J_F(\bc),(2\pi i)\bz)^{G_\br}$ and $H^1(\mathcal J,\co_{\mathcal J})$. Then
\[ \Omega(\mathcal J)=\Omega(\X).\]
\label{neronperiod}\end{lemma}

\begin{proof} We have an isomorphism of $G_\br$-modules $$H^1(\X(\bc),(2\pi i)\bz)\cong H^1(J_F(\bc),(2\pi i)\bz).$$ By \cite{LLR}[Thm.3.1], for each finite place $v$ of $F$ the natural map \cite{LLR}[1.4]
\[ H^1(\X_{S_v},\co_{\X_{S_v}})\to H^1({\mathcal J}_{S_v},\co_{{\mathcal J}_{S_v}})\]
has kernel and cokernel of the same length. From this the statement is immediate.
\end{proof}

We now have all the ingredients to compare (\ref{vanishingatone}) and (\ref{BSDforX}) with the Birch and Swinnerton-Dyer conjecture. Taking cohomology over $\overline{S}$ of the top row in (\ref{genfibre}) gives an exact sequences
\[ 0\to \bigoplus_{v\in\Sigma_f}E_v^{G_{\kappa(v)}}\to \Pic^0(\X)/\Pic(\co_F) \to J_F(F) \to A\to 0\]
where $A$ is a finite abelian group of cardinality $\alpha$ defined in (\ref{alphadef}). Since the $\bz$-rank of $E_v^{G_{\kappa(v)}}$ is $\#C_v-1$ it is clear from (\ref{vancomp}) that (\ref{vanishingatone}) is equivalent to
\[ \ord_{s=1}L(J_F,s)=\rank_\bz J_F(F).\]
By Lemma \ref{innerprod} we find
\begin{align} \prod_{v\in\Sigma_f}\Delta(E_v^{G_{\kappa(v)}})\Delta(J_F(F))=&\Delta(\Pic^0(\X)/\Pic(\co_F))\Delta(A)\notag\\
=&\frac{\Delta(\Pic^0(\X)/\Pic(\co_F))}{\alpha^2}.
\label{deltaid}\end{align}
Applying Prop. \ref{Geisserprop}, (\ref{deltaid}) and Lemma \ref{evintersection} to (\ref{BSDforX}) we find
\begin{align*} \zeta^*(H^1,1)=&\frac{\#\Br(\overline{\X})\cdot\delta^2\cdot\Omega(\X)\cdot R(\X)}{(\#(\Pic^0(\X)_{tor}/\Pic(\co_F)))^2}\cdot\prod\limits_{\text{$v$ real}}\frac{\#\Phi_v}{\delta_v'\delta_v}\\
=&\frac{\prod_v\delta_v'\delta_v}{\alpha^2} \#\Sha(J_F)\cdot\Delta(\Pic^0(\X)/\Pic(\co_F))\cdot\Omega(\X)\cdot\prod\limits_{\text{$v$ real}}\frac{\#\Phi_v}{\delta_v'\delta_v}\\
=&\prod_{v\in\Sigma_f}\left(\frac{\delta_v'\delta_v}{\#\Phi_v}\Delta(E_v^{G_{\kappa(v)}})\right)\cdot\Delta(J_F(F))\cdot  \#\Sha(J_F)\cdot\Omega(\X)\cdot\prod\limits_{v\in\Sigma}\#\Phi_v\\
=&\prod_{v\in\Sigma_f}\left(\log Nv^{|C_v|-1}\prod_{i\in C_v}r_{v,i}\right)\cdot\Delta(J_F(F))\cdot  \#\Sha(J_F)\cdot\Omega(\X)\cdot\prod\limits_{v\in\Sigma}\#\Phi_v
\end{align*}
and (\ref{leadcomp}) and Lemma \ref{neronperiod} show that this identity is equivalent to
\begin{align*} L^*(J_F,1)=&\Delta(J_F(F))\cdot  \#\Sha(J_F)\cdot\Omega(\X)\cdot\prod\limits_{v\in\Sigma}\#\Phi_v\\
=&\frac{R(J_F(F))}{(\#J_F(F)_{tor})^2}\cdot  \#\Sha(J_F)\cdot\Omega(\mathcal J)\cdot\prod\limits_{v\in\Sigma}\#\Phi_v
\end{align*}
 which is the Birch and Swinnerton-Dyer formula. Here we are also using the fact that the Arakelov intersection pairing induces the negative of the Neron-Tate height pairing on $J_F(F)$ by \cite{Hriljac}[Thm.3.1].

\begin{remark} \label{tate} Let
\[\overline{f}:\overline{\X}\to\overline{S}\]
be a flat morphism from a smooth projective surface $\overline{\X}$ to a smooth projective connected curve $\overline{S}\xrightarrow{\pi}\Spec(k)$ over a finite field $k$ and assume $\overline{f}$ has geometrically connected fibres.
 By \cite{bbd}[Rem. 5.4.9] $$K:=R\overline{f}_*\bq_l$$ is a pure complex in the derived category of $l$-adic sheaves on $\overline{S}$ and by \cite{bbd}[Thm. 5.4.5, Thm. 5.4.10] we have
\begin{align*} K\cong & {^p}\mathcal H^1(K)[-1]\oplus{^p}\mathcal H^2(K)[-2]\oplus {^p}\mathcal H^3(K)[-3]\\
\cong & \bq_l[0]\oplus{^p}\mathcal H^2(K)[-2]\oplus \bq_l(-1)[-2] \end{align*}
where ${^p}\mathcal H^i$ refers to the perverse $t$-structure (which agrees with the ordinary $t$-structure over $\Spec(k)$).
Assume in addition that $\overline{f}$ has large monodromy in the sense that $R^0\pi_*R^1\overline{f}_*\bq_l=0$. With notation as in the proof of Lemma \ref{lfunction} we have
\[{^p\mathcal H^2}(K)=(R^1\overline{f}_*\bq_l)[1]\oplus E\otimes\bq_l(-1)[0]\]
and hence our assumption implies $R^{-1}\pi_*{^p\mathcal H^2}(K)=0$ which together with \cite{bbd}[Thm. 5.4.10] gives
$$R^i\pi_*{^p\mathcal H^2}(K)=0; \quad i\neq 0.$$
We find
\begin{align} R^2(\pi\circ\overline{f})_*\bq_l\cong R^2\pi_*(K)\cong &R^2\pi_*\bq_l\oplus R^0\pi_*{^p}\mathcal H^2(K)\oplus R^0\pi_*\bq_l(-1)\notag\\
\cong &\bq_l(-1)\oplus R\pi_*{^p}\mathcal H^2(K)\oplus \bq_l(-1).\label{zetaabs}\end{align}
The Zeta function of the $l$-adic sheaf $R^2(\pi\circ\overline{f})_*\bq_l$ was referred to as $\zeta(H^2_{abs},s)$ in the introduction and the Zeta function of the $l$-adic complex $R\pi_*{^p}\mathcal H^2(K)$ is the function $\zeta(H^1,s)$ discussed in this paper, albeit over the base $\overline{S}$ rather than $S$. By (\ref{zetaabs}) we have
\[ \zeta(H^2_{abs},s)=\zeta(H^1,s)(1-q^{1-s})^2\]
where $q=\#k$. It seems likely that the proof in \cite{gordon},\cite{LLRcor} of the equivalence of the conjecture of Artin and Tate for $\zeta^*(H^2_{abs},1)$ \cite{tate66} with the Birch and Swinnerton-Dyer conjecture for the Jacobian of the generic fibre of $\overline{f}$ can be somewhat simplified using our approach to $\zeta^*(H^1,1)$ via the complex $P^0$.
\end{remark}

\end{document}